\documentclass[12pt]{article}
\usepackage{graphicx} 
\usepackage{amsmath,amssymb,amsthm,mathtools}
\allowdisplaybreaks
\usepackage{enumitem,multicol}
\usepackage[a4paper,margin=1in]{geometry}
\usepackage{setspace}
\usepackage{booktabs}
\usepackage{comment}
\usepackage{float}
\usepackage{subcaption}
\usepackage{listings}
\usepackage{xcolor}
\usepackage{hyperref}
\usepackage[a4paper,margin=1in]{geometry}
\usepackage{authblk}

\theoremstyle{plain}
\newtheorem{theorem}{Theorem}
\newtheorem{lemma}{Lemma}

\newtheorem{proposition}{Proposition}

\theoremstyle{definition}
\newtheorem{definition}{Definition}

\newtheorem{example}{Example}

\title{\textbf{A higher order family of iterative methods for approximating inverse and generalized inverse matrices}}

\author[1]{Alicia Cordero}
\author[1]{Bushra Rani}
\author[1]{Juan R. Torregrosa}

\affil[1]{Instituto Universitario de Matemática Multidisciplinar,\\
Universitat Politècnica de València, València, Spain.\\
E-mail: acordero@mat.upv.es, 
brani@posgrado.upv.es, 
jrtorre@mat.upv.es}

\date{October 3, 2026}

\begin{document}

\maketitle

   \begin{abstract}
       The development of efficient iterative schemes for approximating the inverse and generalized inverse matrices is important in my research fields. To this end, a family of fourth-order matrix iterative schemes for approximating the inverse and generalized inverse matrices is established by incorporating an appropriate weight function into the scalar Heun method. Moreover, a fifth-order matrix iterative scheme is  obtained as a particular member of the proposed family. Under suitable conditions, the theoretical order of convergence is determined for the proposed family, whose stable members are identified through dynamical analysis. To assess their performance and verify the consistency of the theoretical results, some numerical tests are performed on matrices of different sizes.
   \end{abstract}
\section{Introduction}\label{sec:introduction}
In recent decades, computing matrix inverses and generalized inverses has become an important topic of study in many research fields, such as image processing \cite{Chountasis_2009, Chung_2013}, cryptography \cite{Haidary_Makoui_2024}, data analysis \cite{Hawkins_1990}, and signal reconstruction and de-noising \cite{Katsikis_2012}. To solve these problems, many direct methods, such as Gaussian elimination and LU decomposition methods exist, which can be computationally expensive when dealing with large-scale sparse matrices. Consequently, iterative methods have become important tools due to their efficiency and flexibility. The main idea behind these iterative schemes is to generate a sequence of approximations to the exact solution of the problem by repeatedly applying  a computational procedure.
To calculate the inverse of a non-singular square matrix $A \in \mathbb{C}^{N\times N}$, many root-finding iterative methods exist that can be extended to the matrix framework to solve the following nonlinear matrix equation
\begin{equation}\label{eq:non linear matrix eq}
F(X)=X^{-1}-A=0,
\end{equation}
whose unique solution is $X=A^{-1}$ and $F : \mathbb{C}^{N\times N} \to \mathbb{C}^{N\times N}$.

Many iterative schemes for computing matrix inverses and generalized inverses, as well as for analyzing their dynamical behavior, have been introduced in recent years. The development of these methods can be traced back to \cite{Schulz_1933}, where Schulz extended the Newton-Raphson root-finding method to  an  iterative scheme for approximating the inverse of non-singular square matrices. Its mathematical expression is 
\begin{equation}\label{eq:newton_schulz}
X_{n+1} = X_n(2I - AX_n), \qquad n\geq 0,
\end{equation}
which generates a sequence $\{X_n\}$ that converges to $A^{-1}$ under suitable conditions, such as $\|E_0\|=\|I - AX_0\| < 1$  and $A$ being non-singular (see \cite{Higham_2008}). In this case, the residual matrix $E_n = I - AX_n$ satisfies $E_{n+1} = E_n^2$. By taking norms, we obtain $\|E_{n+1}\|\leq \|E_n\|^2$, which implies the quadratic convergence order. In 2007, the same quadratic convergence order was obtained by employing different convergence criteria in \cite{Wu_2007}.

Later, in 2011,  Li et al. extended the Homeier root-finding iterative scheme to the matrix iterative scheme \cite{Li_2011},
\begin{equation}\label{eq:homeier}
X_{n+1}
=
X_n
\left[
I
-
\frac{1}{2}
(I - AX_n)
\left(
I + (2I - AX_n)^2
\right)
\right],
\end{equation}
whose cubic convergence was established in 2023 by Kansal et al. (see \cite{Kansal_2023}).

In general, many classical root-finding methods have been extended to matrix-valued problems in order to construct efficient inversion algorithms,  aiming to increase the convergence order while reducing computational costs. The above-mentioned iterative methods were determined to compute the inverse of a non-singular matrix. To address the need to calculate generalized inverses, such as the Moore-Penrose inverse and the Drazin inverse, many iterative schemes have been established in the literature.

Soleymani et al. introduced a ninth-order method in 2013 for computing the standard inverse and Drazin inverse \cite{Soleymani_2013}, given as
\begin{equation}\label{eq: soleimani et al.}
\begin{aligned}
    X_{n+1}=-\frac{1}{8}&X_n(-7I+9AX_n-5(AX_n)^2+(AX_n)^3)\times(12I-42AX_n+103(AX_n)^2\\&-156(AX_n)^3+157(AX_n)^4-104(AX_n)^5+43(AX_n)^6-10(AX_n)^7+(AX_n)^8).
    \end{aligned}
\end{equation}
Additionally, the authors compared the computational complexity of the proposed method with some existing schemes. Different choices of initial approximations were also proposed in the same manuscript.

Afterward, in 2019, Ahmad et al. proposed a seventh-order scheme for approximating the Drazin inverse of singular matrices in \cite{Ahmad_2019}, given by
\begin{equation}\label{eq: Ahmad et al.}
    X_{n+1}=X_n\Big[I+(AX_n+(AX_n)^2)(I-AX_n+(AX_n)^2)(I+AX_n+(AX_n)^2)\Big].
\end{equation}
Moreover, its efficiency index was computed and compared with some other schemes by employing Ostrowski's definition of the efficiency index \cite{Ostrowski_1966}.
In 2017, a fourth-order method was developed in \cite{Esmaeili_2017} for approximating the Moore-Penrose inverse of a rectangular matrix, which is given by
\begin{equation}\label{eq:stat of art 4th order equation}
    X_{n+1}=X_n\Big[9I-26(AX_n)+34(AX_n)^2-21(AX_n)^3+5(AX_n)^4\Big].
\end{equation}
Subsequently, in 2021, a parametric family of arbitrary order was established in \cite{Cordero_2021} by Cordero et al. for computing the standard and generalized inverses. The authors also determined the basins of attraction of different iterative schemes by employing discrete complex dynamics.

An important limitation of many iterative schemes is their sensitivity to the choice of initial approximation. This has led to increasing interest in the dynamical behavior of iterative methods. In this direction, Cordero et al. introduced a fourth-order family of iterative schemes 
\begin{equation}\label{eq:corder et al. 4th order}
\begin{aligned}
X_{n+1}=X_n\Big[&(4+\alpha)I+AX_n(-2(3+2\alpha)I+AX_n(2(2+3\alpha)I\\&+AX_n(-(1+4\alpha)I+\alpha AX_n)))\Big], \quad n\geq 0, \quad \alpha\in \mathbb{R},
\end{aligned}
\end{equation}
 for computing the standard inverse and the Moore-Penrose inverse and also analyzed the stability of the parametric family using real dynamics in \cite{Cordero_2025}.
Again, in \cite{Cordero_2024}, Cordero et al. introduced a cubic-order Homeier-type family of iterative schemes with a  weight function which is given by
\begin{equation}\label{eq:cordero et al. 3rd order}
    \begin{aligned}
        X_{n+1}=&\frac{1}{2}\frac{X_n(7I-B_n(9I+B_n(-5I+B_n)))}{1+2\gamma} \\&+ \frac{\gamma}{2}\frac{23I-B_n(51I+B_n(-56I+B_n(32I+B_n(-9I+B_n))))}{1+2\gamma},
    \end{aligned}
\end{equation}
where $B_n=AX_n$ and $\gamma \in \mathbb{R} \backslash \{-\frac{1}{2}\}$. Moreover, the stability behavior of different members of the above family was discussed by using real dynamics. Next, some basic notions related to the generalized inverse and the real dynamics of the iterative schemes are given below.
\section{Preliminaries on the iterative schemes and their dynamics}\label{sec:preliminaries}
\begin{definition}[Q-order of convergence for scalar iterative schemes]
Let $\{x_n\}$ be a sequence converging to $\bar{x}$. The sequence is said
to converge with order $q\geq 1$ if there exists a constant
$\mu>0$ such that
\[
\lim_{n\to\infty}
\frac{|e_{n+1}|}{|e_n|^q}
=
\mu.
\]
The constant $\mu$ is called the asymptotic error constant. The value of $q$ determines the speed of convergence. If $q=1$ and $0<\mu<1$, the convergence is \emph{linear}. If $q=2$, the convergence is \emph{quadratic}. If $q=3$, the convergence is \emph{cubic} (see \cite{Weerakoon_2000}).
\end{definition}
\begin{definition}[Q-order of convergence for matrix iterative schemes]
    Following the previous definition of convergence order, a matrix iterative method is said to have a  convergence order of $q$ if there exists a constant $C>0$ such that
\[
\|E_{n+1}\| \leq C\|E_n\|^q,
\]
for sufficiently large values of $n$. The value of $q$ illustrates the order of convergence. (see \cite{Higham_2008}).
\end{definition}

In some cases, the classical inverse of matrices does not exist in the usual sense. The generalized inverses are used to handle such problems. For rectangular or singular matrices, it is essential to  extend the concept of the classical matrix inverse into the generalized inverse. A detailed  study of generalized inverses can be found in \cite{Ben-Israel_2003, Wang_2018}. We give below some important types of generalized inverses, which will be used later in the study of iterative schemes.
\begin{definition}[Moore-Penrose inverse]
The Moore-Penrose inverse or pseudoinverse $X$ of a rectangular matrix $A\in \mathbb{C}^{M\times N}$ is defined as the unique matrix satisfying the following Penrose equations \cite{Ben-Israel_2003}:
\begin{multicols}{4}
\begin{enumerate}[label=\roman*.]
    \item $AX A = A$, 
    \item $X A X=X$,
    \item $AX = (AX)^*$, 
    \item $X A=(X A)^*$.
\end{enumerate}
\end{multicols}
For rectangular matrices, the Moore-Penrose inverse is a natural extension of the classical inverse. The Moore-Penrose inverse of a matrix $A$  is denoted by $A^\dagger$.
\end{definition}
\begin{definition}[Drazin inverse]
Let $A \in \mathbb{C}^{N\times N}$ be a square complex matrix.
The index of $A$ is denoted by
$\operatorname{ind}(A),$ 
 which is defined as the smallest non-negative integer \(k\) satisfying,
$\operatorname{rank}(A^{k+1})
=
\operatorname{rank}(A^k).$
The Drazin inverse of $A$, denoted by $A^D$, is the unique matrix satisfying the following conditions \cite{Ben-Israel_2003}:
\begin{multicols}{3}
    \begin{enumerate}[label=(\alph*)]
    \item $A^{k+1}A^D=A^k$,
    \item $A^DAA^D=A^D$,
    \item $AA^D=A^DA$.
\end{enumerate}
\end{multicols}

Here, $k$ is the index of $A$. The Drazin inverse is useful when the matrix $A$ is singular or non-invertible.
\end{definition}
Next, some concepts from the dynamics of the iterative schemes to analyze their stability are defined. Following the references \cite{Argyros_2017, Chicharro_2013, Devaney_2018},  some basic notions of dynamical analysis are described below.

Let $R:\widehat{\mathbb{C}}\to\widehat{\mathbb{C}}$ be a rational operator obtained by applying an iterative scheme to a quadratic polynomial defined on the Riemann sphere $\widehat{\mathbb{C}}=\mathbb{C}\cup \{\infty\}$.
\begin{definition}[Orbit of a point]
    Let $R:\widehat{\mathbb{C}}\to\widehat{\mathbb{C}}$ be a rational operator and let $z_0\in\widehat{\mathbb{C}}$.  
The orbit of $z_0$ under the action of $R$ is the sequence
\begin{equation*}
\mathcal{O}(z_0)=\{z_0,R(z_0),R^2(z_0),\ldots,R^n(z_0),\ldots\}.
\end{equation*}
\end{definition}
\begin{definition}[Fixed point]
A point $z^*\in\widehat{\mathbb{C}}$ is called a fixed point of the rational operator $R$ if
\begin{equation*}
R(z^*)=z^*.
\end{equation*}
\end{definition}
\begin{definition}[Strange fixed point]
A fixed point of the rational operator that is not a root of the polynomial equation $p(z)=0$ is called a strange fixed point.  
\end{definition}
   \begin{definition}[Periodic point]
A point $\hat{z}\in\widehat{\mathbb{C}}$ is called a periodic point of period $k \in \mathbb{N}$ if
\begin{equation*}
R^k(\hat{z})=\hat{z},
\end{equation*}
and $\;R^p(\hat{z})\neq \hat{z}, \; \forall\, p\in\mathbb{N},\; 1\leq p<k.$

\end{definition}
    \begin{definition}[Stability of a fixed point]
Let $z^*$ be a fixed point of the rational operator $R$. The dynamical behavior of $z^*$ is determined by the modulus of the derivative of $R(z^*)$:
\begin{itemize}
    \item \textbf{Attracting}: if $|R'(z^*)|<1$.
    \item \textbf{Repelling}: if $|R'(z^*)|>1$.
    \item \textbf{Parabolic (or Neutral)}: if $|R'(z^*)|=1$.
    \item \textbf{Superattracting}: if $|R'(z^*)|=0$.
\end{itemize}
\end{definition}
\begin{definition}[Critical point]
A point $z^*\in\widehat{\mathbb{C}}$ is called a critical point of the rational operator $R$ if
\begin{equation*}
R'(z^*)=0.
\end{equation*}
\end{definition}
    \begin{definition}[Free critical point]
Critical points that are not roots of the polynomial equation $p(z)=0$ are called free critical points.
\end{definition}

\begin{definition}[Basin of attraction]
Let $z^*$ be an attracting fixed point of the rational operator $R$.  
Then the basin of attraction of $z^*$ is defined as:
\begin{equation*}
A(z^*)=\left\{z_0\in\widehat{\mathbb{C}}:\lim_{n\to\infty}R^n(z_0)=z^*\right\}.
\end{equation*}
\end{definition}

    \begin{definition}[Immediate basin of attraction]
The immediate basin of attraction of an attractor $z^*$ is the connected component of $A(z^*)$ that contains $z^*$.
\end{definition}

\begin{theorem}[Fatou--Julia]\label{theo:Fatou and julia}
Let $R$ be a rational function.  The immediate basin of attraction of each periodic orbit contains at least one critical point.
\end{theorem}

In this work, we build upon these contributions and basic notions by proposing a fourth-order family of matrix iterative schemes. In Section \ref{sec:proposed scalar iterative scheme}, a fourth-order root-finding family of iterative schemes is constructed using a weight function. Its matrix formulation is discussed in Section  \ref{sec: Family of Heun-type matrix iterative schemes} along with its convergence analysis for approximating the standard inverse, Moore-Penrose inverse, and Drazin inverse. The dynamical analysis of the newly developed iterative family is established in Section \ref{sec: dynamics}. Finally, Section \ref{sec:numerical} presents some numerical experiments to study different numerical behaviors.
 
 \section{Proposed iterative scheme}\label{sec:proposed scalar iterative scheme}

In \cite{Singh_2014}, Singh et al. introduced a two-step, third-order class of iterative schemes for an initial guess $x_0\in \mathbb{R}$.
 The method is defined as follows:
\begin{equation}\label{eq:general scheme3}
\begin{aligned}
y_n &= x_n-a\frac{f(x_n)}{f'(x_n)},\\
x_{n+1} &= x_n-A(t_n)\frac{f(x_n)}{f'(x_n)}, \qquad n\geq 0,
\end{aligned}
\end{equation}
where $t_n=\frac{f'(y_n)}{f'(x_n)}  \; \text{and}
\; a\in\mathbb{R}$.
The weight function $A(t_n)$ must satisfy the following conditions
\[
A(1)=1,
\qquad
A'(1)=-\frac{1}{2a},
\qquad
|A''(1)|<\infty.
\]
To obtain the third-order iterative method known as the Heun scheme, a  particular choice was considered
\[
A(t_n)=\frac{t_n+3}{4t_n},
\qquad
a=\frac{2}{3}.
\]
So, the iterative scheme \eqref{eq:general scheme3} becomes
\begin{equation}
\begin{aligned}
y_n
&=
x_n-\frac{2}{3}\frac{f(x_n)}{f'(x_n)},
\\
x_{n+1}
&=
x_n
-
\frac{f(x_n)}{4}
\left(
\frac{1}{f'(x_n)}
+
\frac{3}{f'(y_n)}
\right),
\end{aligned}
\label{eq:Heun_orignal_scalar}
\end{equation}
and its error equation was defined as follows
\begin{equation*}
e_{n+1}
=
\frac{2}{3}c_2^2e_n^3
+
\left(
\frac{5}{3}c_2c_3
-
\frac{13}{9}c_2^3
+
\frac{1}{9}c_4
\right)e_n^4
+
O(e_n^5),
\end{equation*}
where $
c_m=\frac{f^{(m)}(\bar{x})}{m!f'(\bar{x})},
\;
m\geq 2$. This shows the cubic order of convergence of the Heun method.

Next, we incorporate a new weight function into the Heun root-finding scheme defined in \eqref{eq:Heun_orignal_scalar}. The purpose of introducing this weight function is to increase the convergence order without increasing the number of functional evaluations per iteration. Hence, the new Heun-type root-finding scheme is
\begin{equation}
\label{eq:Heun4_scaler}
\begin{aligned}
y_n &= x_n-\frac{2}{3}\frac{f(x_n)}{f'(x_n)}, \\
x_{n+1} &= x_n-
G(\eta_n)\,
\frac{f(x_n)}{4}
\left(
\frac{1}{f'(x_n)}
+\frac{3}{f'(y_n)}
\right), \qquad n\geq 0,
\end{aligned}
\end{equation}
where $\eta_n=\frac{f'(x_n)}{f'(y_n)}$ 
and the weight function $G(\eta_n)$ must satisfy the following conditions
\[
G(1)=1,\qquad
G'(1)=0,\qquad
G''(1)=\frac{3}{4},
\quad \text{and} \quad
|G^{(3)}(1)|<\infty.
\]
The following result establishes the order of convergence of the proposed iterative scheme under the above conditions.
\begin{theorem}
\label{theo:Heun4_scalar_convergence}
Let $\bar{x}$ be a simple root of a sufficiently differentiable function $f$ in a neighborhood of $\bar{x}$, and consider $x_0$ to be an initial guess close enough to the root. Then, the iterative scheme defined in \eqref{eq:Heun4_scaler} has fourth-order convergence provided that the weight function $G(\eta_n)$ satisfies
\[
G(1)=1,\qquad
G'(1)=0,\qquad
G''(1)=\frac{3}{4},
\quad \text{and} \quad
|G^{(3)}(1)|<\infty.
\]
\end{theorem}

\begin{proof}
Consider $e_n=x_n-\bar{x},$ to be the local error at the $n$-th iteration. By expanding $f(x_n)$  about the simple root $\bar{x}$ using the Taylor series, we have 
\begin{equation*}
f(x_n)=f'(\bar{x})e_n+\frac{f''(\bar{x})}{2!}e_n^2+\frac{f^{(3)}(\bar{x})}{3!}e_n^3+\frac{f^{(4)}(\bar{x})}{4!}e_n^4+\frac{f^{(5)}(\bar{x})}{5!}e_n^5+\dots.
    \label{eq: taylor basic2}
    \end{equation*}
Define 
\[
c_m=\frac{f^{(m)}(\bar{x})}{m!f'(\bar{x})},
\qquad
m\geq 2.
\]
Hence, we obtain
\begin{equation}
f(x_n)
=
f'(\bar{x})
\left[
e_n
+
c_2e_n^2
+
c_3e_n^3
+
c_4e_n^4
+
c_5e_n^5
+
O(e_n^6)
\right],
\label{eq:Heun3_taylor}
\end{equation}
and
\begin{equation}
f'(x_n)
=
f'(\bar{x})
\left[
1
+
2c_2e_n
+
3c_3e_n^2
+
4c_4e_n^3
+
5c_5e_n^4
+
O(e_n^5)
\right].
\label{eq: taylor_f'(x_n)}
\end{equation}
Dividing \eqref{eq:Heun3_taylor} by \eqref{eq: taylor_f'(x_n)} and simplifying, we obtain
\begin{equation}
\frac{f(x_n)}{f'(x_n)}
=
e_n
-
c_2e_n^2
+
2(c_2^2-c_3)e_n^3
+
(7c_2c_3-3c_4-4c_2^3)e_n^4
+
O(e_n^5).
\label{eq:Heun_taylor_division}
\end{equation}
Substituting  \eqref{eq:Heun_taylor_division} into the predictor step of \eqref{eq:Heun4_scaler}, we have
\begin{equation*}
y_n
=
\bar{x}+\frac{1}{3}e_n
+\frac{2}{3}c_2e_n^2
-\frac{4}{3}(c_2^2-c_3)e_n^3 -\frac{2}{3}
\left(
7c_2c_3-3c_4+4c_2^3
\right)e_n^4
+O(e_n^5),
\label{eq:y_n taylor}
\end{equation*}
As a result, we obtain the error equation of the first step of \eqref{eq:Heun_orignal_scalar} is 
\begin{equation}
\begin{aligned}
e_{y_n}
=
\frac{1}{3}e_n
&+
\frac{2}{3}c_2e_n^2
-
\frac{4}{3}(c_2^2-c_3)e_n^3
-
\frac{2}{3}
(7c_2c_3-3c_4+4c_2^3)e_n^4
+
O(e_n^5).
\end{aligned}
\label{eq:Heun_taylor_error_scaler1}
\end{equation}
Next, expanding \(f'(y_n)\) about \(\bar{x}\) and substituting \eqref{eq:Heun_taylor_error_scaler1}, we obtain
\begin{equation}
\begin{aligned}
\label{eq:taylor_f'(y_n)}
f'(y_n)
=
f'(\bar{x})
\Big[
1
&+
\frac{2}{3}c_2e_n
+
\left(
\frac{4}{3}c_2^2+\frac{1}{3}c_3
\right)e_n^2+
\left(
\frac{4}{3}c_2c_3
+
\frac{8}{3}(c_2c_3-c_2^3)
+
\frac{4}{27}c_4
\right)e_n^3
+
O(e_n^4)
\Big].
\end{aligned}
\end{equation}
From \eqref{eq: taylor_f'(x_n)} and \eqref{eq:taylor_f'(y_n)}, we get
\begin{equation}
\begin{aligned}
\eta_n
=
\frac{f'(x_n)}{f'(y_n)}
=
1
&+\frac{4}{3}c_2e_n
+\frac{4}{9}(-5c_2^2+6c_3)e_n^2 +\frac{8}{27}
(8c_2^3-21c_2c_3+13c_4)e_n^3
+O(e_n^4).
\end{aligned}
\label{eq: eta_n taylor}
\end{equation}
Expanding the weight function $G(\eta_n)$ about $\eta_n=1$ yields
\begin{equation}
\label{eq:Taylor_weight_fun}
\begin{aligned}
G(\eta_n)
=
&G(1)
+G'(1)(\eta_n-1)+\frac{1}{2}G''(1)(\eta_n-1)^2+\frac{1}{6}G^{(3)}(1)(\eta_n-1)^3
+O((\eta_n-1)^4).
\end{aligned}
\end{equation}
From \eqref{eq:Heun3_taylor}, \eqref{eq: taylor_f'(x_n)}, and \eqref{eq:taylor_f'(y_n)}, we obtain
\begin{equation}
\begin{aligned}
\label{eq:Heun4_subpart_corrector}
\frac{1}{4}
\left(
\frac{f(x_n)}{f'(x_n)}
+
\frac{3f(x_n)}{f'(y_n)}
\right)
=
e_n
+\left(
\frac{7}{6}c_2^2-\frac{1}{2}c_3
\right)e_n^3 +\left(
\frac{5}{3}c_2c_3
-\frac{13}{9}c_2^3
+\frac{1}{9}c_4
\right)e_n^4 +O(e_n^5).
\end{aligned}
\end{equation}
Substituting \eqref{eq:Taylor_weight_fun} and \eqref{eq:Heun4_subpart_corrector} into the corrector step of \eqref{eq:Heun4_scaler} and simplifying gives
\begin{equation*}
e_{n+1}
=
\frac{1}{81}
\Big[
(63-4G^{(3)}(1))c_2^3-81c_2c_3
+4c_4
\Big]e_n^4
+O(e_n^5).
\label{eq: heun scalar 4th error eq}
\end{equation*}
Hence, the method possesses fourth-order convergence.
\end{proof}
Next, we define some weight functions that satisfy all the conditions of the above theorem and also possess fourth-order convergence. The objective of all these functions is to choose only those functions that can construct an inverse-free matrix iterative scheme.

\subsection{Particular choice of weight functions}
Consider the weight function
\begin{align*}
G(\eta_n)
=
\frac{1}{1+3\eta_n}
\Bigg[
&\beta
+\left(
\frac{15}{2}-3\beta
\right)\eta_n +(-6+3\beta)\eta_n^2 +\left(
\frac{5}{2}-\beta
\right)\eta_n^3
\Bigg],
\qquad
\beta\in\mathbb{R},
\label{eq: weight fun. 1}
\end{align*}
which  satisfies all conditions of the 
Theorem \ref{theo:Heun4_scalar_convergence}. Hence, the iterative scheme becomes
\begin{equation}
\label{eq:Heun4_scaler_final}
\begin{aligned}
y_n
&=
x_n-\frac{2}{3}\frac{f(x_n)}{f'(x_n)}, \\
x_{n+1}
=
x_n-
\frac{1}{1+3\eta_n}
\Bigg[
&\beta
+\left(
\frac{15}{2}-3\beta
\right)\eta_n +(-6+3\beta)\eta_n^2 \\&+\left(
\frac{5}{2}-\beta
\right)\eta_n^3
\Bigg]
\left(
\frac{f(x_n)}{4f'(x_n)}
+
\frac{3f(x_n)}{4f'(y_n)}
\right),
\end{aligned}
\end{equation}
where $\eta_n=\frac{f'(x_n)}{f'(y_n)}$  and the corresponding error equation is
\begin{equation*}\label{eq: error eq heun with weight fun. 1}
e_{n+1}
=
\frac{1}{27}
\Big[
(-1+16\beta)c_2^3
-27c_2c_3
+3c_4
\Big]e_n^4
+O(e_n^5),
\end{equation*}
which confirms that the family of Heun-type  iterative schemes achieves fourth-order convergence. An appropriate choice of weight function is required to achieve an inverse free matrix iterative scheme. The formulation of an inverse-free matrix iterative scheme cannot be achieved by an arbitrary weight function that satisfies the conditions of Theorem \ref{theo:Heun4_scalar_convergence}.
\section{Family of Heun-type matrix iterative schemes}\label{sec: Family of Heun-type matrix iterative schemes}
In this section, we extend the scalar fourth-order schemes defined in \eqref{eq:Heun4_scaler_final}  to matrix iterative schemes for approximating the inverse of a non-singular matrix $A\in\mathbb{C}^{N\times N}$.
The matrix formulation can be derived directly from the scalar equation
\begin{equation}\label{eq:scalar to matrix transf}
f(x)=\frac{1}{x}-a.
\end{equation}
After simplification, the resulting scalar expressions can be converted into matrix form.

 The matrix formulation of \eqref{eq:Heun4_scaler_final} is as follows:
\begin{equation}
\label{eq:Heun4_matrix}
\begin{aligned}
X_{n+1}
=
X_n+ \frac{1}{729} X_n \Big[&16 \beta I (-4I + A X_n)^3 (-I + A X_n)^4 - (5I - 2 A X_n)^2 (-I+ 
       A X_n) \\&\times(205I + A X_n (-5I + A X_n) (71I + 10 A X_n (-5I + A X_n)))\Big].
\end{aligned}
\end{equation}
Equation \eqref{eq:Heun4_matrix} represents the inverse-free  family of Heun-type iterative methods for approximating the inverse matrices. This expression can be written as,

\begin{equation}\label{eq: simplified HEUN4}
X_{n+1} = \frac{1}{729} X_n \, \mathcal{P}(A X_n),
\end{equation}
where 
\begin{align*}
\mathcal{P}(A X_n)&=
5854I-1024\beta I -18100AX_n
+4864\beta AX_n+28920(AX_n)^2
-9408\beta(AX_n)^2 \\
&-26285(AX_n)^3
+9488\beta(AX_n)^3 +13874(AX_n)^4
-5312\beta(AX_n)^4 -4134(AX_n)^5\\&
+1632\beta(AX_n)^5 +640(AX_n)^6
-256\beta(AX_n)^6 -40(AX_n)^7
+16\beta(AX_n)^7.
\end{align*}

Throughout this work, we employ the inverse-free family defined in \eqref{eq:Heun4_matrix} as the Heun-type family of matrix iterative schemes.
We denote this family of iterative schemes by $BR(\beta)$. The next step is to investigate the theoretical convergence properties of the newly derived family of matrix iterative schemes.
\subsection{Convergence analysis for standard inverse}

\begin{theorem}\label{theo:convergence_heun4}
Let $A \in \mathbb{C}^{N \times N}$ be a non-singular matrix, and let the initial approximation $X_0$ satisfy $\|I - AX_0\| < 1$.  If $\beta \in \left[\frac{25}{16}, \frac{5}{2}\right]$, then the sequence $\{X_n\}$ generated by the iterative scheme \eqref{eq:Heun4_matrix} converges to $A^{-1}$ with order $p=4$.  Moreover, the  error satisfies 
\begin{align*}
       \|e_{n+1}\|\leq \|A^{-1}\|\,\|A\|^4\|e_n\|^4&\Bigg[\left|-\frac{25}{27}+ \frac{16\beta}{27}\right| +\frac{10}{27}\|A\|\,\|e_n\| + \left|\frac{86}{81}-\frac{32\beta}{81}\right|\|A\|^2\|e_n\|^2\\& + \left|\frac{320}{729}-\frac{128\beta}{729}\right| \|A\|^3\,\|e_n\|^3 + \left|\frac{40}{729}-\frac{16\beta}{729}\right|\|A\|^4
       \|e_n\|^4 \Bigg],
    \end{align*}
where $e_n = X_n - A^{-1}$. In particular, for $\beta=\frac{25}{16}$, the iterative scheme ${X_n}$ converges to inverse matrix $A^{-1}$ with convergence order $p=5$. Moreover, the error satisfies
\begin{equation*}
\|e_{n+1}\|\leq\|A^{-1}\|\,\|A\|^5\|e_n\|^5\Bigg[\frac{10}{27} + \frac{4}{9}\|A\|\,\|e_n\| + \frac{120}{729} \|A\|^2\|e_n\|^2 + \frac{15}{729}\|A\|^3
       \|e_n\|^3 \Bigg],
\end{equation*}
\end{theorem}
\begin{proof}
Consider the residual matrix at the $n$-th iteration $E_n = I - AX_n.$ At the $(n+1)$-th iteration,  the residual matrix satisfies
\begin{equation}
\label{eq:residual_n+1}
E_{n+1}
= I - AX_{n+1}.
\end{equation}
Substituting \eqref{eq:Heun4_matrix} into the equation \eqref{eq:residual_n+1} yields
\begin{align*}
E_{n+1}
= \frac{1}{729}\Big[&
(5854-1024\beta)AX_n - (18100-4864\beta)(AX_n)^2 + (28920-9408\beta)(AX_n)^3 \\&-(26285-9488\beta)(AX_n)^4 + (13874-5312\beta)(AX_n)^5 - (4134-1632\beta)(AX_n)^6 \\
&+ (640-256\beta)(AX_n)^7 - (40-16\beta)(AX_n)^8
\Big] \; - I.
\end{align*}
Since $AX_n = I - E_n$, substitution gives a polynomial expansion in $E_n$:
\begin{equation}\label{eq:residual_heun4_without_norm}
E_{n+1}
=
\left(-\frac{25}{27}+\frac{16\beta}{27}\right)E_n^4
+ \frac{10}{27}E_n^5
+ \left(\frac{86}{81}-\frac{32\beta}{81}\right)E_n^6 
+ \left(\frac{320}{729}-\frac{128\beta}{729}\right)E_n^7
+ \left(\frac{40}{729}-\frac{16\beta}{729}\right)E_n^8.
\end{equation}
Taking norms on both sides yields
\begin{equation}
\begin{aligned}
\|E_{n+1}\|
\leq&
\left|-\frac{25}{27}+\frac{16\beta}{27}\right|\|E_n\|^4
+ \frac{10}{27}\|E_n\|^5 + \left|\frac{86}{81}-\frac{32\beta}{81}\right|\|E_n\|^6
\\&+ \left|\frac{320}{729}-\frac{128\beta}{729}\right|\|E_n\|^7 + \left|\frac{40}{729}-\frac{16\beta}{729}\right|\|E_n\|^8.
\end{aligned}
\end{equation}
Since $\|E_0\| < 1$,  by induction process $\|E_n\| < 1$ for $n\geq 0$. Hence, by defining
\[
w =
\left|-\frac{25}{27}+\frac{16\beta}{27}\right|
+ \frac{10}{27}
+ \left|\frac{86}{81}-\frac{32\beta}{81}\right|
+ \left|\frac{320}{729}-\frac{128\beta}{729}\right|
+ \left|\frac{40}{729}-\frac{16\beta}{729}\right|,
\]
we obtain
\begin{align*}
\|E_{n+1}\|
&\leq w \|E_n\|^4 \\
&\leq w (w \|E_{n-1}\|^4)^4
= w^{1+4} \|E_{n-1}\|^{4^2} \\
&\leq w^{1+4+4^2} \|E_{n-2}\|^{4^3}\\
&\quad  \vdots\\
& \leq w^{\frac{4^{n+1}-1}{3}}\|E_0\|^{4^{n+1}} ,\quad\quad \forall n \geq 0.
\end{align*}
For $\beta \in \left[\frac{25}{16}, \frac{5}{2}\right]$, a direct evaluation shows that $w = 1$. Because $\|E_0\|<1$, then $\|E_0\|^{4^{n+1}} \to 0$ when $n \to \infty$. Consequently $\|E_{n+1}\| \leq \|E_0\|^{4^{n+1}}$.

Hence,  $\|E_n\| \to 0 \;\; \text{as} \;n \to \infty$, 
which implies $X_n \to A^{-1}.$ So, the convergence of the sequence $\{X_n\}$ is  established.

Additionally, we have
\begin{align*}
\|I-AX_{n+1}\|
\leq\|I-AX_{n}\|^4&\Bigg( \left|-\frac{25}{27}+ \frac{16\beta}{27}\right| +\frac{10}{27}\|I-AX_{n}\| + \left|\frac{86}{81}-\frac{32\beta}{81}\right|\|I-AX_{n}\|^2 \\&+ \left|\frac{320}{729}-\frac{128\beta}{729}\right| \|I-AX_{n}\|^3 + \left|\frac{40}{729}-\frac{16\beta}{729}\right|\|I-AX_{n}\|^4\Bigg),
\end{align*}
and we can write the above expression as
\begin{align*}
\|A(A^{-1}-X_{n+1})\|
\leq\|A(A^{-1}-X_{n})\|^4&\Bigg( \left|-\frac{25}{27}+ \frac{16\beta}{27}\right| +\frac{10}{27}\|A(A^{-1}-X_{n})\| \\&+ \left|\frac{86}{81}-\frac{32\beta}{81}\right|\|A(A^{-1}-X_{n})\|^2 + \\&\left|\frac{320}{729}-\frac{128\beta}{729}\right| \|A(A^{-1}-X_{n})\|^3+ \\&\left|\frac{40}{729}-\frac{16\beta}{729}\right|\|A(A^{-1}-X_{n})\|^4\Bigg).
\end{align*}
Therefore,
\begin{align*}
\|A^{-1}A(A^{-1}-X_{n+1})\|
&\le \|A^{-1}\|\,\|A(A^{-1}-X_{n+1})\| \\
&\le \|A^{-1}\|\,\|A(A^{-1}-X_{n})\|^4 \Bigg( 
\left|-\frac{25}{27}+ \frac{16\beta}{27}\right|  + \frac{10}{27}\|A(A^{-1}-X_{n})\|\\&+ \left|\frac{86}{81}-\frac{32\beta}{81}\right|\|A(A^{-1}-X_{n})\|^2 + \left|\frac{320}{729}-\frac{128\beta}{729}\right|\|A(A^{-1}-X_{n})\|^3\\&  + \left|\frac{40}{729}-\frac{16\beta}{729}\right|\|A(A^{-1}-X_{n})\|^4 
\Bigg)  \\
&\le \left|-\frac{25}{27}+ \frac{16\beta}{27}\right|
\|A^{-1}\|\,\|A\|^4 \|A^{-1}-X_{n}\|^4 \\&+\frac{10}{27}\|A^{-1}\|\,\|A\|^5 \|A^{-1}-X_{n}\|^5 \\& + \left|\frac{86}{81}-\frac{32\beta}{81}\right|
\|A^{-1}\|\,\|A\|^6 \|A^{-1}-X_{n}\|^6 \\& + \left|\frac{320}{729}-\frac{128\beta}{729}\right|
\|A^{-1}\|\,\|A\|^7 \|A^{-1}-X_{n}\|^7 \\&+ \left|\frac{40}{729}-\frac{16\beta}{729}\right|
\|A^{-1}\|\,\|A\|^8 \|A^{-1}-X_{n}\|^8.
\end{align*}
The error between the exact and approximated solutions is defined as $e_n=X_n-A^{-1}$. Then,
    \begin{align*}
       \|e_{n+1}\|\leq& \left|-\frac{25}{27}+ \frac{16\beta}{27}\right| \|A^{-1}\|\,\|A\|^4\|e_n\|^4+\frac{10}{27}\|A^{-1}\|\,\|A\|^5\|e_n\|^5 \\&+\left|\frac{86}{81}-\frac{32\beta}{81}\right|\|A^{-1}\|\,\|A\|^6\|e_n\|^6 + \left|\frac{320}{729}-\frac{128\beta}{729}\right| \|A^{-1}\|\,\|A\|^7\|e_n\|^7 \\&+\left|\frac{40}{729}-\frac{16\beta}{729}\right|\|A^{-1}\|\,\|A\|^8
       \|e_n\|^8. 
    \end{align*}
After factorization, we obtain
\begin{equation}\label{eq:Error_eq_heun4}
    \begin{aligned}
       \|e_{n+1}\|\leq \|A^{-1}\|\,\|A\|^4\|e_n\|^4&\Bigg(\left|-\frac{25}{27}+ \frac{16\beta}{27}\right| +\frac{10}{27}\|A\|\,\|e_n\| + \left|\frac{86}{81}-\frac{32\beta}{81}\right|\|A\|^2\|e_n\|^2\\& + \left|\frac{320}{729}-\frac{128\beta}{729}\right| \|A\|^3\,\|e_n\|^3 + \left|\frac{40}{729}-\frac{16\beta}{729}\right|\|A\|^4
       \|e_n\|^4 \Bigg).
    \end{aligned}
\end{equation}

Hence, the sequence $\{X_n\}$ converges to the inverse matrix $A^{-1}$ with order of convergence at least four for $\frac{25}{16} \leq \beta \leq \frac{5}{2}$.

In particular, by choosing $\beta=\frac{25}{16}$, factoring \eqref{eq:Error_eq_heun4} gives
\begin{equation*}
\|e_{n+1}\|\leq\|A^{-1}\|\,\|A\|^5\|e_n\|^5\Bigg[\frac{10}{27} + \frac{4}{9}\|A\|\,\|e_n\| + \frac{120}{729} \|A\|^2\|e_n\|^2 + \frac{15}{729}\|A\|^3
       \|e_n\|^3 \Bigg],
\end{equation*}
which leads to the fifth-order convergence of the iterative scheme for approximating the standard inverse of $A$.
\end{proof}
\subsection{Convergence analysis for generalized inverses}
Some auxiliary results are presented below to establish the convergence of the proposed iterative schemes for approximating the Moore--Penrose inverse of a rectangular matrix.

\begin{lemma}[{\cite{Stanimirovic_2008}}]
\label{lemma:spectral_radius_only}  Consider $P \in \mathbb{C}^{N\times N}$ and $Q\in \mathbb{C}^{N\times N}$ to be the complex square matrices, which satisfy $P=P^2$ and $PQ=QP$, then
    \[\rho(PQ)\leq \rho(Q).\]
\end{lemma}
\begin{lemma}[{\cite{Horn_1985}}]
\label{lemma:spectral_radius_norm} 
Let $A \in \mathbb{C}^{N\times N}$ and $\epsilon >0$. Then the matrix norm satisfies
\[\rho(A)\leq \|A\|\leq \rho(A)+\epsilon,\]
where $\rho(A)$ denotes the largest absolute value of the eigen values of $A$, known as spectral radius of $A$.
\end{lemma}
The above lemmas will help us establish the convergence analysis of the Heun-type iterative methods for calculating the Moore-Penrose inverse of rectangular matrices. To achieve the same purpose, another result is provided below.
\begin{lemma}
Let $\{X_n\}$ be the sequence generated by \eqref{eq:Heun4_matrix}, where 
$\beta \in \left[\frac{25}{16}, \frac{5}{2}\right]$, and the initial approximation is given by
$X_0 = \alpha A^*$, where $\alpha \in \mathbb{R}$ is a real constant. Then, the following properties hold for all $n \ge 0$:
\begin{multicols}{2}
\begin{enumerate}[label=i.]
    \item $X_n A A^\dagger = X_n$,
    \item $A^\dagger A X_n = X_n$,
    \item $(A X_n)^* = A X_n$,
    \item $(X_n A)^* = X_n A$.
\end{enumerate}
\end{multicols}
\end{lemma}
\begin{proof}
We prove \textup{(ii)} by induction. The proof of \textup{(i)} is analogous.\\
For $n=0$
\[
A^\dagger A X_0
= A^\dagger A (\alpha A^*)
= \alpha (A^\dagger A) A^*.
\]
Using $(A^\dagger A)^* = A^\dagger A$ and properties of the Moore–Penrose inverse,
\[
\alpha (A^\dagger A)^* A^*
= \alpha A^* (A^\dagger)^* A^*
= \alpha (A A^\dagger A)^*
= \alpha A^*
= X_0.
\]
Next, we assume that $A^\dagger A X_n = X_n$. We show that $A^\dagger A X_{n+1} = X_{n+1}$.\\
From \eqref{eq: simplified HEUN4},
\[
A^\dagger AX_{n+1} = \frac{1}{729} A^\dagger A X_n \, \mathcal{P}(A X_n),
\]
where 
\begin{align*}
\mathcal{P}(A X_n)&=
5854I-1024\beta I -18100AX_n
+4864\beta AX_n+28920(AX_n)^2
-9408\beta(AX_n)^2 \\
&-26285(AX_n)^3
+9488\beta(AX_n)^3 +13874(AX_n)^4
-5312\beta(AX_n)^4 -4134(AX_n)^5\\&
+1632\beta(AX_n)^5 +640(AX_n)^6
-256\beta(AX_n)^6 -40(AX_n)^7
+16\beta(AX_n)^7
.
\end{align*}
Then,
\begin{align*}
A^\dagger A X_{n+1}
&= \frac{1}{729} X_n \, \mathcal{P}(A X_n) \quad  \\
&= X_{n+1}.
\end{align*}
Hence, property (ii) holds for all $n \ge 0$.

Next, we use induction process to prove (iii).

For $n=0$,
\[
(AX_0)^* = (\alpha A A^*)^* = \alpha A A^* = AX_0.
\]
Assume $(AX_n)^* = AX_n$. Then,
\begin{align*}
(AX_{n+1})^*
&= \left(\frac{1}{729} A X_n \mathcal{P}(A X_n)\right)^* \\
&= \frac{1}{729} \mathcal{P}(A X_n)^* (A X_n)^*.
\end{align*}
Since $(AX_n)^* = AX_n$ and $\mathcal{P}(AX_n)$ is a polynomial in $AX_n$, all its powers are Hermitian under the induction assumption, hence
\[
(AX_{n+1})^* = AX_{n+1}.
\]
For (iv), the proof follows similarly by taking adjoints on the right-hand side and using the induction hypothesis, yielding
\[
(X_{n+1}A)^* = X_{n+1}A.
\]
\end{proof}
Thus, all four identities are preserved under the iteration, and we are now able to establish the convergence and the convergence order of the Heun-type matrix iterative methods for approximating the Moore–Penrose inverse.
\begin{theorem}\label{theo:Moore_pen_Heun4}
Let $A \in \mathbb{C}^{M \times N}$ be a rectangular complex matrix. Consider $\{X_n\}$ as a sequence generated by the iterative schemes \eqref{eq:Heun4_matrix}, where the initial approximation is chosen as $X_0 = \alpha A^*,$ with $\alpha < \frac{1}{\sigma_1^2}$, and let $\beta \in \left[\frac{25}{16}, \frac{5}{2}\right].$ 
Then, the sequence $\{X_n\}$ converges to the Moore-Penrose inverse $A^\dagger$ with order $p=4$. Additionally, the error satisfies
\begin{equation*}
\|e_{n+1}\|
\leq 
\|A^\dagger\| \, \|A\|^4 \, \|e_n\|^4.
\end{equation*}
In particular, for $\beta=\frac{25}{16}$, the iterative scheme ${X_n}$ converges to the Moore-Penrose inverse $A^{\dagger}$ with convergence order $p=5$, and the error satisfies
\begin{equation*}
\|e_{n+1}\|
\leq 
\|A^\dagger\| \, \|A\|^5 \, \|e_n\|^5,
\end{equation*}
where $e_n = X_n - A^\dagger$.
\end{theorem}
\begin{proof}
Define $e_n = X_n - A^\dagger$ as the error matrix at the $n$-th iteration, and define the residual matrix by $E_n = I - AX_n.$
Then,
\begin{align*}
Ae_{n+1}
&= AX_{n+1} - AA^\dagger \\
&= -E_{n+1} + I - AA^\dagger.
\end{align*}
Using \eqref{eq:residual_heun4_without_norm}, we obtain
\begin{align*}
Ae_{n+1}
&=
\left(\frac{25}{27}-\frac{16\beta}{27}\right)E_n^4
-\frac{10}{27}E_n^5
-\left(\frac{86}{81}-\frac{32\beta}{81}\right)E_n^6\\&
-\left(\frac{320}{729}-\frac{128\beta}{729}\right)E_n^7
-\left(\frac{40}{729}-\frac{16\beta}{729}\right)E_n^8
+ I - AA^\dagger.
\end{align*}
Since
\[
E_n = I - AX_n
      = (I-AA^\dagger) - Ae_n,
\]
it follows that
\begin{equation}\label{eq:Huen4_MP_before_iden_4}
\begin{aligned}
Ae_{n+1}
={}&
\left(\frac{25}{27}-\frac{16\beta}{27}\right)
\big((I-AA^\dagger)-Ae_n\big)^4
-\frac{10}{27}
\big((I-AA^\dagger)-Ae_n\big)^5\\&
-\left(\frac{86}{81}-\frac{32\beta}{81}\right)
\big((I-AA^\dagger)-Ae_n\big)^6
-\left(\frac{320}{729}-\frac{128\beta}{729}\right)
\big((I-AA^\dagger)-Ae_n\big)^7
\\&-\left(\frac{40}{729}-\frac{16\beta}{729}\right)
\big((I-AA^\dagger)-Ae_n\big)^8
+ I - AA^\dagger.
\end{aligned}
\end{equation}
Using the relation (see \cite{Cordero_2021})
\[
(I-AA^\dagger)^k = I-AA^\dagger,
\qquad
(I-AA^\dagger)Ae_n = 0,
\qquad
\forall k \in \mathbb{N},
\]
the above expression simplifies considerably. In particular,
\[
\big((I-AA^\dagger)-Ae_n\big)^q
=
(I-AA^\dagger) + (-1)^q(Ae_n)^q,
\]
for $q=4,5,6,7,8$.

Substituting these identities into the  equation \eqref{eq:Huen4_MP_before_iden_4} yields
\begin{align*}
Ae_{n+1}
={}&
\left(\frac{25}{27}-\frac{16\beta}{27}\right)(Ae_n)^4
+\frac{10}{27}(Ae_n)^5
-\left(\frac{86}{81}-\frac{32\beta}{81}\right)(Ae_n)^6
\\&+\left(\frac{320}{729}-\frac{128\beta}{729}\right)(Ae_n)^7
-\left(\frac{40}{729}-\frac{16\beta}{729}\right)(Ae_n)^8.
\end{align*}
Taking norms on both sides, we obtain
\begin{equation} \label{eq:MP_error_bound_4}
\begin{aligned}
\|Ae_{n+1}\|
\leq {} &
\|Ae_n\|^4
\Bigg(
\left|\frac{25}{27}-\frac{16\beta}{27}\right|
+\frac{10}{27}\|Ae_n\|
+\left|\frac{86}{81}-\frac{32\beta}{81}\right|
\|Ae_n\|^2
\\&+\left|\frac{320}{729}-\frac{128\beta}{729}\right|
\|Ae_n\|^3
+\left|\frac{40}{729}-\frac{16\beta}{729}\right|
\|Ae_n\|^4
\Bigg).
\end{aligned}
\end{equation}
Now consider $P=AA^\dagger$ and $Q=I-AX_0$. Hence,
\[P^2=(AA^\dagger)(AA^\dagger)=AA^\dagger=P.\]
Moreover,
\begin{equation}
    \begin{aligned}
        PQ&=(AA^\dagger)(I-AX_0)\\&
        =AA^\dagger-AA^\dagger AX_0\\&
        =AA^\dagger-AX_0\\&
        =AA^\dagger-AX_0AA^\dagger\\&
        =(I-AX_0)AA^\dagger\\&
        =QP.
     \label{eq:PQ=QP}
     \end{aligned}
\end{equation}
From Lemma \ref{lemma:spectral_radius_only} and equation \eqref{eq:PQ=QP}, we obtain
    \begin{align*}
        \rho((A A ^\dagger) (I-AX_0))&\leq\rho(I-AX_0)\\&
        =\rho(I-\alpha AA^*)\\&
        =\max_{1\leq n \leq q}\lvert1-\lambda_n(\alpha A A^*) \rvert
        < 1.
    \end{align*}
where $q$ denotes the number of eigenvalues of the matrix $\alpha A A^*.$

From Lemma \ref{lemma:spectral_radius_norm}, given $\epsilon >0$,  there exists a matrix norm $\|.\|$ that satisfies the following inequality:
\[
\|A(A^\dagger-X_0)\|\leq \rho(A(A^\dagger-X_0))+\epsilon <1.
\]
Hence, $\|Ae_0\|<1.$ It is easy to prove that $\|Ae_n\|<1.$ Then, from \eqref{eq:MP_error_bound_4}, we have
\begin{equation*}
\|Ae_{n+1}\|
\leq 
\|Ae_n\|^4
\Bigg(
\left|\frac{25}{27}-\frac{16\beta}{27}\right|
+\frac{10}{27}
+\left|\frac{86}{81}-\frac{32\beta}{81}\right|
+\left|\frac{320}{729}-\frac{128\beta}{729}\right|
+\left|\frac{40}{729}-\frac{16\beta}{729}\right|
\Bigg).
\end{equation*}
Define
\begin{equation*}
w
=
\left|\frac{25}{27}-\frac{16\beta}{27}\right|
+\frac{10}{27}
+\left|\frac{86}{81}-\frac{32\beta}{81}\right|
+\left|\frac{320}{729}-\frac{128\beta}{729}\right|
+\left|\frac{40}{729}-\frac{16\beta}{729}\right|.
\end{equation*}
Hence,
\[
\|Ae_{n+1}\| \leq w \|Ae_n\|^4.
\]

For $\beta \in \left[\frac{25}{16},\frac{5}{2}\right],$
we obtain $w=1$. Therefore, we can write
\[
\|Ae_{n+1}\| \leq \|Ae_n\|^4.
\]
Finally, using the identity
\[
e_{n+1}=A^\dagger Ae_{n+1},
\]
we obtain
\[
\|e_{n+1}\|
\leq
\|A^\dagger\|\,\|Ae_{n+1}\|.
\]
Moreover,
\[
\|Ae_n\| \leq \|A\|\,\|e_n\|.
\]
Consequently,
\begin{equation*}
\|e_{n+1}\|
\leq 
\|A^\dagger\| \, \|A\|^4 \, \|e_n\|^4.
\end{equation*}
Hence,
\[
\|e_{n+1}\| \to 0
\;\;\text{as} \;\; n\to\infty,
\]
and therefore the sequence $\{X_n\}$ converges to the Moore--Penrose inverse $A^\dagger$ with fourth-order convergence.

For $\beta=\frac{25}{16}$,  equation \eqref{eq:Heun4_matrix} becomes
\begin{equation}
\begin{aligned}
    X_{n+1}=&\frac{1}{243}X_n (1418 + A X_n (-3500 
      +A X_n (4740+ A X_n (-3820  \\&+A X_n (1858 + A X_n (-528 - 5 A X_n (-16 + A X_n))))))).
      \end{aligned}
\end{equation}
by following the same procedure, we obtain the error inequality
\begin{equation*}
\|e_{n+1}\|
\leq 
\|A^\dagger\| \, \|A\|^5 \, \|e_n\|^5.
\end{equation*}
\end{proof}
 To determine the convergence properties of the iterative scheme defined in equation \eqref{eq:Heun4_matrix} for approximating the Drazin inverse, we first present several auxiliary results that will be useful in the convergence analysis.
\begin{proposition}[{\cite{Artidiello_2020}}]
\label{proposition:projector}
Let \(P_{A,B}\) be the projector onto a subspace \(A\) along a subspace \(B\). Then:
\begin{enumerate}[label=\roman*.]
    \item $ P_{A,B}C=C$  if and only if
   $\mathcal{R}(C)\subseteq A,$
    \item $CP_{A,B}=C$ if and only if $\mathcal{N}(C)\supseteq B.$
\end{enumerate}
\end{proposition}
The above proposition is helpful for establishing the convergence analysis of the Heun iterative scheme for approximating the Drazin inverse matrix. 
\begin{theorem}\label{theo:Drazin_Heun4}
Let $A \in \mathbb{C}^{N \times N}$ be a singular square matrix, and let $\{X_n\}$ be the sequence generated by the iterative scheme \eqref{eq:Heun4_matrix}, with the initial approximation $X_0$, and 
$\beta \in \left[\frac{25}{16},\frac{5}{2}\right]$. The sequence $\{X_n\}$ converges to the Drazin inverse $A^D$ with fourth-order convergence. Moreover, the sequence generated by $\{X_n\}$ satisfies the error estimate for computing the Drazin inverse, which is
 \begin{align*}
\|e_{n+1}\|
\leq &
\|A^D\|\,\|A\|^4\,\|e_n\|^4
\Bigg(
\left|-\frac{25}{27}+\frac{16\beta}{27}\right|
+\frac{10}{27}\|A\|\|e_n\|
+\left|\frac{86}{81}-\frac{32\beta}{81}\right|
\|A\|^2\|e_n\|^2
\\
&\quad
+\left|\frac{320}{729}-\frac{128\beta}{729}\right|
\|A\|^3\|e_n\|^3
+\left|\frac{40}{729}-\frac{16\beta}{729}\right|
\|A\|^4\|e_n\|^4
\Bigg).
\end{align*}
In particular, for $\beta=\frac{25}{16}$, the sequence $\{X_n\}$ converges to $A^{D}$ with order $p=5$, and $\{X_n\}$ must satisfy the error estimate for approximating the Drazin inverse of $A$
\begin{align*}
\|e_{n+1}\|
\leq {} &
\|A^D\| \, \|A\|^5 \, \|e_n\|^5
\Bigg(
\frac{10}{27}
+\frac{4}{9}\|A\|\|e_n\|
+\frac{120}{729}\|A\|^2\|e_n\|^2
+\frac{15}{729}\|A\|^3\|e_n\|^3
\Bigg),
\end{align*}
\end{theorem}
\begin{proof}
Let A be a $N \times N$ singular matrix, let $k=\operatorname{ind}(A)$ and $r=\operatorname{rank}(A^k)$; then there exists a non-singular matrix $V$ such that 
\[
A=V\begin{pmatrix}
    P &0\\0 & Q
\end{pmatrix}V^{-1},
\]
where $P$ is a non-singular matrix of size $r \times r$ and $Q$ is a nilpotent matrix of size $(N-r) \times (N-r)$ (see \cite{Ben-Israel_2003}). Its Drazin inverse is 
\[
A^D=V\begin{pmatrix}
    P^{-1} &0\\0 & 0
\end{pmatrix}V^{-1}.
\]
By choosing the initial approximation $X_0=\alpha A^l$, where $l \geq k=ind(A)$, we obtain
\[
X_0=V\begin{pmatrix}
    \alpha P^l &0\\0 & \alpha Q^l
\end{pmatrix}V^{-1}=V\begin{pmatrix}
    X_{0,P} &0 \\0&0
\end{pmatrix}V^{-1},
\]
which shows that the initial approximation $X_0$ preserves the structure of Drazin inverse.  Suppose that 
\[
X_n=V\begin{pmatrix}
     X_{n,P} &0\\0 & 0
\end{pmatrix}V^{-1},
\]
then
\[
AX_n=V\begin{pmatrix}
     PX_{n,P} &0\\0 & 0
\end{pmatrix}V^{-1},
\]
where $P$ is a non-singular matrix. For every power $q$, we have
\[
(AX_n)^q=V\begin{pmatrix}
     (PX_{n,P})^q &0\\0 & 0
\end{pmatrix}V^{-1}.
\]
From \eqref{eq: simplified HEUN4}, we can write
\[\mathcal{P}(A X_n)=V\begin{pmatrix}
    \mathcal{P}(PX_{n,P})&0\\0 &  \frac{5854-1024\beta}{729}I
\end{pmatrix},\]
where 
\begin{align*}
\mathcal{P}(A X_n)=&
5854I_r-1024\beta I_r -18100PX_{n,P}
+4864\beta PX_n+28920(PX_{n,P})^2
\\&-9408\beta(PX_{n,P})^2-26285(PX_{n,P})^3
+9488\beta(PX_{n,P})^3 +13874(PX_{n,P})^4
\\&-5312\beta(PX_{n,P})^4 -4134(PX_{n,P})^5
+1632\beta(PX_{n,P})^5 +640(PX_{n,P})^6
\\&-256\beta(PX_{n,P})^6 -40(PX_{n,P})^7
+16\beta(PX_{n,P})^7.
\end{align*}

Then the final expression of equation \eqref{eq: simplified HEUN4} becomes
\[X_{n+1}=\begin{pmatrix}
    X_{n,P}\mathcal{P}(PX_{n,P})&0\\0 &  0
    \end{pmatrix},\]
    where $X_{n,P}\mathcal{P}(A_1X_{n,P})$ is a non-singular matrix that can be written as
        \begin{align*}
        X_{n+1,P}=\frac{1}{729}X_{n,P}\Big[&5854I_r-1024\beta I_r -18100PX_{n,P}
+4864\beta PX_n\\&+28920(PX_{n,P})^2
-9408\beta(PX_{n,P})^2-26285(PX_{n,P})^3
\\&+9488\beta(PX_{n,P})^3 +13874(PX_{n,P})^4
-5312\beta(PX_{n,P})^4 \\&-4134(PX_{n,P})^5
+1632\beta(PX_{n,P})^5 +640(PX_{n,P})^6\\&
-256\beta(PX_{n,P})^6 -40(PX_{n,P})^7
+16\beta(PX_{n,P})^7\Big].
        \end{align*}
   By considering $E_{n,P}=I_r-PX_{n,P}$, the above expression can be written as
\begin{equation}
\begin{aligned}
E_{n+1,P}
=&
\left(-\frac{25}{27}+\frac{16\beta}{27}\right)E_{n,P}^4
+ \frac{10}{27}E_{n,P}^5
+ \left(\frac{86}{81}-\frac{32\beta}{81}\right)E_{n,P}^6 
\\&+ \left(\frac{320}{729}-\frac{128\beta}{729}\right)E_{n,P}^7
+ \left(\frac{40}{729}-\frac{16\beta}{729}\right)E_{n,P}^8.
\end{aligned}
\end{equation}
By choosing the initial approximation $X_0$ as $\|E_{0,P}\|=\|I_r-PX_{n+1,P}\|\leq 1$, it is easy to prove $\|E_{n,P}\|\leq 1,$ then we have
\begin{equation*}
\|E_{n+1,P}\|\leq w\|E_{n,P}\|^4\leq 
w^{\frac{4^n-1}{3}}\|E_{0,P}\|^{4^{n}} ,\quad\quad \forall n \geq 0,
\end{equation*}
where 
\[
w =
\left|-\frac{25}{27}+\frac{16\beta}{27}\right|
+ \frac{10}{27}
+ \left|\frac{86}{81}-\frac{32\beta}{81}\right|
+ \left|\frac{320}{729}-\frac{128\beta}{729}\right|
+ \left|\frac{40}{729}-\frac{16\beta}{729}\right|.
\]
For $\beta \in \left[\frac{25}{16}, \frac{5}{2}\right]$, a direct evaluation shows that $w = 1$. Because we choose $\|E_{0,P}\|<1$, then $\|E_{0,P}\|^{4^{n+1}} \to 0$ when $n \to \infty$. Consequently $\|E_{n+1,P}\| \leq \|E_{0,P}\|^{4^{n+1}}$.

Hence,  $\|E_{n,P}\| \to 0 \;\; \text{as} \;n \to \infty$, 
which implies $X_{n,P} \to P^{-1}.$
Hence, the convergence of the sequence $\{X_n\}$ is  established.
Since $X_0=\alpha A^l$, where $l \geq k=ind(A)$, we obtain
\[
\mathcal{R}(X_0)=\mathcal{R}(A^l)\subseteq\mathcal{R}(A^k).
\]
Therefore, we can state that
\[
\mathcal{R}(X_n)\subseteq\mathcal{R}(X_{n-1})\subseteq \mathcal{R}(X_{n-2})\subseteq\dots\subseteq\mathcal{R}(X_{0})\subseteq \mathcal{R}(A^k),
\]
and consequently,
\begin{equation}
\mathcal{R}(X_n)\subseteq\mathcal{R}(A^k),
\qquad
\forall n\geq0.
\label{eq:range_drazin_4}
\end{equation}
Similarly
\begin{equation}
\mathcal{N}(X_0)\supseteq\mathcal{N}(A^k),
\qquad
\forall n\geq0.
\end{equation}
Assume that
\[t \in \mathcal{N}(A^k),\]
then, after simplification
\[X_0t=\alpha A^l t =0.\]
This property is preserved by the iteration if
\[X_nt=0, \qquad AX_nt=0.\]
From \eqref{eq: simplified HEUN4}, we deduce $X_{n+1}t=0$. Consequently,
\[t \in \mathcal{N}({X_{n+1}}).\]
By induction,
\begin{equation}
    \label{eq:null_space_drazin_4}
    \mathcal{N}(A^k) \subseteq \mathcal{N}(X_n), \qquad n \geq 0.
\end{equation}
From \cite{Ben-Israel_2003}, we have
\begin{equation}
AA^D=A^DA=
P_{\mathcal{R}(A^k),\mathcal{N}(A^k)},
\label{eq:projector_RN_drazin_4}
\end{equation}
where $P_{\mathcal{R}(A^k),\mathcal{N}(A^k)}$ 
denotes the projector onto $\mathcal{R}(A^k)$ along $\mathcal{N}(A^k)$.\\
By employing Proposition \ref{proposition:projector} together with equations \eqref{eq:range_drazin_4}, \eqref{eq:null_space_drazin_4}, and \eqref{eq:projector_RN_drazin_4}, we conclude that
\begin{equation}\label{eq:Drazin_lemma_proof_1}
X_nAA^D=X_n=A^DAX_n,
\qquad
\forall n\geq0.
\end{equation}
Let $e_n = X_n - A^D$
denote the error matrix at the $n$-th iteration. Then
\begin{equation*}
\begin{aligned}
    e_n&=X_n-A^D\\&
    =A^D(I-AX_n)\\&
    =A^DE_n.
\end{aligned}
\end{equation*}
Taking norms on both sides and simplifying yields
\begin{equation}
    \|e_n\|\leq \|A^D\|\,\|E_n\|\leq w^{\frac{4^n-1}{3}}\|A^D\|\,\|E_0\|^{4^{n}} ,\quad\quad \forall n \geq 0. 
\end{equation}
We can write this as
\[\|A^D-X_n\|\leq w^{\frac{4^n-1}{3}}\|A^D\|\,\|I-AX_0\|^{4^{n}} ,\quad\quad \forall n \geq 0.\]
Hence, the expression stated in the theorem is satisfied. From the initial condition, we have assumed that $\|I-AX_0\|<1$. Then, $\|I-AX_0\|^{4^n} \to 0$ as $n \to \infty$. 
Hence, $\|A^D-X_n\|\to 0,$ and consequently $X_n \to A^D.$ So,  the convergence is proven. 
Moreover,
\begin{align*}
Ae_{n+1}
&= AX_{n+1} - AA^D \\
&= -E_{n+1} + I - AA^D.
\end{align*}
Using \eqref{eq:residual_heun4_without_norm}, we obtain
\begin{align*}
Ae_{n+1}
&={}
\left(\frac{25}{27}-\frac{16\beta}{27}\right)E_n^4
-\frac{10}{27}E_n^5
-\left(\frac{86}{81}-\frac{32\beta}{81}\right)E_n^6
\\
&
-\left(\frac{320}{729}-\frac{128\beta}{729}\right)E_n^7
-\left(\frac{40}{729}-\frac{16\beta}{729}\right)E_n^8
+ I - AA^D.
\end{align*}
Since
\[
E_n = I - AX_n
      = (I-AA^D)-Ae_n,
\]
it follows that
\begin{equation}
\begin{aligned}\label{eq:Ae(n+1)_drazin4}
Ae_{n+1}
&=
\left(\frac{25}{27}-\frac{16\beta}{27}\right)
\big((I-AA^D)-Ae_n\big)^4
-\frac{10}{27}
\big((I-AA^D)-Ae_n\big)^5
\\
&
-\left(\frac{86}{81}-\frac{32\beta}{81}\right)
\big((I-AA^D)-Ae_n\big)^6
-\left(\frac{320}{729}-\frac{128\beta}{729}\right)
\big((I-AA^D)-Ae_n\big)^7
\\
&
-\left(\frac{40}{729}-\frac{16\beta}{729}\right)
\big((I-AA^D)-Ae_n\big)^8
+ I - AA^D.
\end{aligned}
\end{equation}
Using the identities (see \cite{SoleymaniF_2013})
\[
(I-AA^D)^m = I-AA^D,
\qquad
(I-AA^D)Ae_n = 0,
\qquad
\forall m \in \mathbb{N},
\]
we obtain
\[
\big((I-AA^D)-Ae_n\big)^p
=
(I-AA^D)+(-1)^p(Ae_n)^p,
\]
for $p=4,5,6,7,8$.

Substituting these expressions into the equation \eqref{eq:Ae(n+1)_drazin4} gives
\begin{align*}
Ae_{n+1}
&={}
\left(\frac{25}{27}-\frac{16\beta}{27}\right)(Ae_n)^4
+\frac{10}{27}(Ae_n)^5
-\left(\frac{86}{81}-\frac{32\beta}{81}\right)(Ae_n)^6
\\
&
+\left(\frac{320}{729}-\frac{128\beta}{729}\right)(Ae_n)^7
-\left(\frac{40}{729}-\frac{16\beta}{729}\right)(Ae_n)^8.
\end{align*}
Taking norms on both sides yields
\begin{align*}
\|Ae_{n+1}\|
\leq {} &
\|Ae_n\|^4
\Bigg(
\left|\frac{25}{27}-\frac{16\beta}{27}\right|
+\frac{10}{27}\|Ae_n\|
+\left|\frac{86}{81}-\frac{32\beta}{81}\right|
\|Ae_n\|^2
\\
&\quad
+\left|\frac{320}{729}-\frac{128\beta}{729}\right|
\|Ae_n\|^3
+\left|\frac{40}{729}-\frac{16\beta}{729}\right|
\|Ae_n\|^4
\Bigg).
\end{align*}
Since
\[
e_{n+1}=A^DAe_{n+1},
\]
we have
\[
\|e_{n+1}\|
\leq
\|A^D\|\,\|Ae_{n+1}\|.
\]
Moreover,
\[
\|Ae_n\| \leq \|A\|\,\|e_n\|.
\]
Therefore,
\begin{equation}\label{eq:error_Drazin4}
\begin{aligned}
\|e_{n+1}\|
\leq &
\|A^D\|\,\|A\|^4\,\|e_n\|^4
\Bigg(
\left|-\frac{25}{27}+\frac{16\beta}{27}\right|
+\frac{10}{27}\|A\|\|e_n\|
+\left|\frac{86}{81}-\frac{32\beta}{81}\right|
\|A\|^2\|e_n\|^2
\\
&\quad
+\left|\frac{320}{729}-\frac{128\beta}{729}\right|
\|A\|^3\|e_n\|^3
+\left|\frac{40}{729}-\frac{16\beta}{729}\right|
\|A\|^4\|e_n\|^4
\Bigg).
\end{aligned}
\end{equation}

Hence, $\|e_n\| \to 0
\;\;\text{as}\;\; n\to\infty$, which implies that $X_n \to A^D.$\\
Therefore, for $\beta \in \left[\frac{25}{16},\frac{5}{2}\right] $, the sequence $\{X_n\}$ converges to the Drazin inverse $A^D$ with fourth-order convergence.
For $\beta=\frac{25}{16},$ the equation \eqref{eq:error_Drazin4} becomes 
\begin{align*}
\|e_{n+1}\|
\leq {} &
\|A^D\| \, \|A\|^5 \, \|e_n\|^5
\Bigg(
\frac{10}{27}
+\frac{4}{9}\|A\|\|e_n\|
+\frac{120}{729}\|A\|^2\|e_n\|^2
+\frac{15}{729}\|A\|^3\|e_n\|^3
\Bigg),
\end{align*}
which represents the fifth-order of convergence.
\end{proof}
\section{Dynamical analysis of Heun-type matrix fourth-order scheme} \label{sec: dynamics}
Next, we study the dynamical behavior of the fourth-order Heun-type parametric family defined by the equation \eqref{eq:Heun4_matrix}.

Unlike the scalar iterative schemes, for the dynamical study of matrix schemes, we use a rational function
\begin{equation}\label{eq:rational_fun_4}
f(x)=\frac{1}{x}-1,
\end{equation}
which is appropriate for the qualitative analysis of matrix iterative schemes for approximating the  inverse of non-singular matrices.  Next, we apply this rational function to the parametric family of matrix iterative schemes to obtain the associated operator. 
\begin{proposition}
Consider the  rational function
\[
f(x)=\frac{1}{x}-1.
\]
The operator associated with the fourth-order Heun-type family of matrix iterative schemes defined in  \eqref{eq:Heun4_matrix} is given by
\begin{equation}\label{eq:operator_matix4}
R_\beta(x)=x+\frac{1}{729}x(x-1)(16\beta(4-5x+x^2)^3-(5-2x)^2(205-355x+321x^2-100x^3+10x^4)),
\end{equation}
which is defined for all $\beta\in\mathbb{R}$.
\end{proposition}
\begin{proof}
 Substituting the rational function \eqref{eq:rational_fun_4} 
   into the scalar form associated with the parametric Heun-type fourth-order matrix iterative schemes and simplifying yields the expression \eqref{eq:operator_matix4}.
\end{proof}
Since a simpler form of the Heun-type family of matrix iterative schemes has been established in terms of the operator $R_\beta(x)$. Moreover, it should be noted that, just as in the case of the parametric Heun-type matrix iterative schemes, the operator $R_\beta(x)$ also depends on the parameter.
Next, we use this operator to determine its fixed points, their stability, and its critical points. In matrix iterative schemes, the dynamical analysis is generally carried out on the real axis. However, we restrict our study to real-valued dynamics and do not consider complex-valued results.  
\subsection{Fixed points of \texorpdfstring{$R_\beta(x)$}{Rbeta(x)}}
\begin{theorem}
Let $R_\beta:D\subset\mathbb{R}\to\mathbb{R}$ be the operator associated with the family of Heun-type matrix schemes applied to \eqref{eq:rational_fun_4}.
Then, $x=1$ is a fixed point of $R_\beta(x)$ for all $\beta\in\mathbb{R}$. Moreover, $x=0$ is a strange fixed point for all $\beta\in\mathbb{R}$.  Additionally:
\begin{enumerate}
    \item If $ \beta\in(-\infty,0]\cup\left(\frac52,\infty\right),$
    then the operator $R_\beta(x)$ has  two real strange fixed points, denoted by $ex_i(\beta)$, $i=1,2$, which are the roots of the polynomial
\begin{equation}
 \begin{aligned}
q(x)=&-5125+1024\beta+(12975-3840\beta)x+(-15945+5568\beta)x^2\\&+(10340-3920\beta)x^3+(-3534+1392\beta)x^4+(600-240\beta)x^5\\&+(-40+16\beta)x^6.
\end{aligned}
\end{equation}
    In particular:
\begin{itemize}
  \item For $\beta=0$,  
$ex_1(\beta)=ex_2(\beta)=\frac{5}{2}$.
\item For $\beta=\frac{5125}{1024}$,  $ex_1(\beta)=0,\;\text{and} \; ex_2(\beta)=5.$
    \end{itemize}

    \item If $\beta\in\left(0,\frac52\right],$
    then no strange fixed points exist, except $x=0$.
\end{enumerate}
\end{theorem}
\begin{proof}
To determine the fixed points, we solve
\[
R_\beta(x)=x.
\]
After simplification, we have
\begin{equation}\label{eq:fixed_point_exp_matrix4}
\frac{1}{729}x(x-1)\left(16\beta(4-5x+x^2)^3-(5-2x)^2(205-355x+321x^2-100x^3+10x^4)\right)=0.
\end{equation}
The above expression has two real roots 
\[
x=0,\qquad x=1,
\]
and the real roots of the polynomial $q(x)$.

Since $x=1$ is the root of the rational function $f(x)$, so it is considered a fixed point, meanwhile the remaining real roots of \eqref{eq:fixed_point_exp_matrix4} are classified as the strange fixed points. The intervals for $\beta$ stated in the theorem indicates the parameter values for which the polynomial $q(x)$ has real roots.
\end{proof}

The root of the function, $x=1$, is a superattracting fixed point for every $\beta\in\mathbb{R}$. Next, we determine the stability of the strange fixed points.
\subsubsection{Stability of the fixed points}
To determine the stability of the strange fixed points, we establish their attracting, repelling, parabolic, and superattracting behaviors.
\begin{theorem}
The strange fixed point $x=0$ has attracting behavior for  $\beta\in\left(\frac{5125}{1024},\frac{6585}{1024}\right)$, superattracting behavior at $\beta=\frac{5854}{1024},$ parabolic behavior at $\beta=\frac{5125}{1024}
        \;\; \text{and} \;\;
        \beta=\frac{6585}{1024},$ and repelling behavior for all other real values of $\beta$.
\end{theorem}
\begin{proof}
To determine the stability of the strange fixed point $x=0$, we compute the derivative of the associated rational operator $R_\beta(x)$, which is
\begin{equation}
R'_\beta(x)=\frac{2}{729}(x-1)^3(-2927+9319x-6642x^2+1760x^3-160x^4+32\beta(x-4)^2(1-6x+2x^2)).
\end{equation}
\noindent
Evaluating $R'_\beta(x)$ at $x=0$, we obtain
\begin{equation*}
R'_\beta(0)=\frac{2}{729}(2927-512\beta).
\end{equation*}
The condition for a point to be attracting is
\[
|R'_\beta(0)|<1.
\]
After simplifying, we obtain
\[
\frac{5125}{1024}<\beta<\frac{6585}{1024},
\]
which is the attracting region for the strange fixed point $x=0$ .

Next,
\[
|R'_\beta(0)|=0,
\]
gives
\[
\beta=\frac{5854}{1024},
\]
which determines the superattracting parameter value for $x=0$.
Similarly,
\[
|R'_\beta(0)|=1,
\]
yields
\[
\beta=\frac{5125}{1024},
\qquad
\beta=\frac{6585}{1024},
\]
correspond to the parameter value for which $x=0$ is parabolic.

Hence, for all remaining values of $\beta$, the strange fixed point $x=0$ is repelling.
\begin{figure}[H]
    \centering
\includegraphics[width=0.5\textwidth]{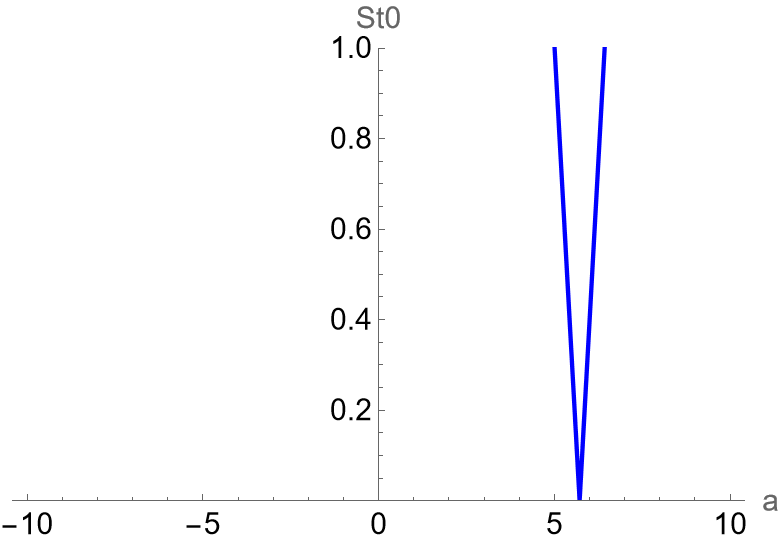}
    \caption{Stability region of the strange fixed point $x=0$.}
    \label{fig:PF_Heun4_Matrix_x0}
\end{figure}

Figure~\ref{fig:PF_Heun4_Matrix_x0} represents the stability region of $x=0$, which also validates the previous theoretical results. 
\end{proof}
\begin{theorem}
The stability of the strange fixed point $ex_1(\beta)$ and $ex_2(\beta)$ is given as follows:
\begin{enumerate}
    \item For
    $\beta\in(-\infty,0],$
    \begin{itemize}
        \item the strange fixed point $ex_1(\beta)$ is always repulsive,
        \item the strange fixed point $ex_2(\beta)$ is superattracting at
 $\beta\approx -0.06403$, attracting for $\beta\in(-0.207,-1.02\times10^{-7})$, parabolic at
            $\beta=-0.207
            \; \text{and} \;
            \beta=-1.02\times10^{-7}$, and repelling elsewhere.
    \end{itemize}

    \item For
   $ \beta\in\left[\frac52,\infty\right),$
    \begin{itemize}
        \item the strange fixed point $ex_1(\beta)$ is 
         superattracting at
            $\beta\approx4.4727$,
           attracting for
           $ \beta\in(4.154,5.005)$, parabolic at
           $ \beta=4.154
            \; \text{and} \;
            \beta=5.005$, and  repulsive elsewhere,
        \item $ex_2(\beta)$ is always repulsive.
    \end{itemize}
    \end{enumerate}
\end{theorem}
\begin{proof}
It is difficult to calculate the stability region of  $ex_1(\beta)$ and $ex_2(\beta)$ analytically. Therefore, to justify the regions stated in the  theorem, we  analyze and represent them graphically.

\textbf{Stability of the fixed point \texorpdfstring{$ex_1(\beta)$}}.
\begin{figure}[H]
    \centering
    \includegraphics[width=0.5\textwidth]{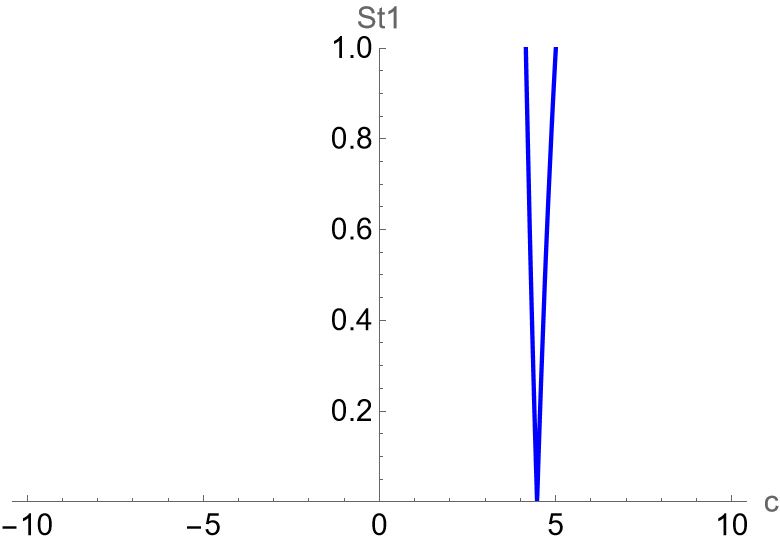}
    \caption{$|R'_\beta(ex_1(\beta))|<1$ for $\beta\in(-\infty,0] \cup[\frac52,\infty)$.}
    \label{fig:PF_Heun4_Matrix_ex1R}
\end{figure}

For $\beta\in(-\infty,0]$, Figure~\ref{fig:PF_Heun4_Matrix_ex1R} contains no points, which means that the fixed point $ex_1(\beta)$ is repulsive there. meanwhile, for $\beta\in[\frac52,\infty)$, the parameter region where $ex_1(\beta)$ becomes attracting is located approximately in the interval $\beta \in [4, 5]$.

\textbf{Stability of the fixed point \texorpdfstring{$ex_2(\beta)$}}.

\begin{figure}[H]
    \centering
    \includegraphics[width=0.5\textwidth]{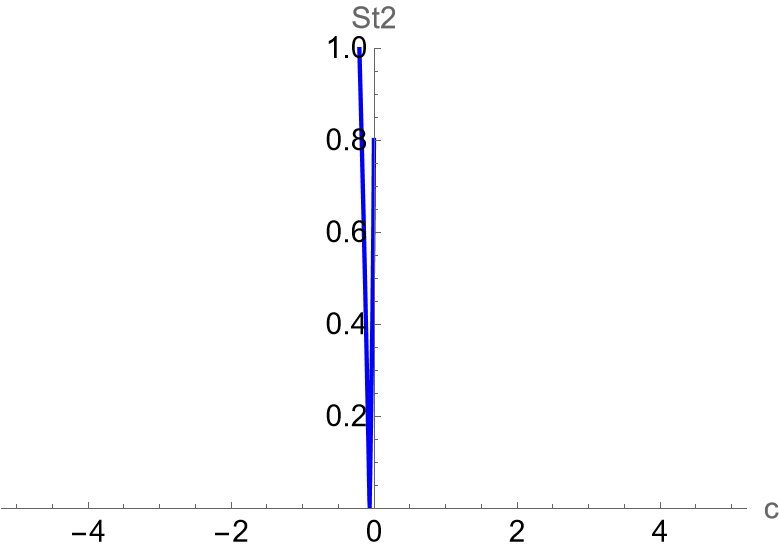}
    \caption{$|R'_\beta(ex_2(\beta))|<1$ for $\beta\in(-\infty,0] \cup[\frac52,\infty)$.}
    \label{fig:PF_Heun4_Matrix_ex2A}
\end{figure}
Figure~\ref{fig:PF_Heun4_Matrix_ex2A} represents the parameter region where $ex_2(\beta)$ becomes attracting,   which is approximately located in the interval $[-0.2, 0]$. Moreover, $ex_2(\beta)$ remains repulsive in the interval $\beta\in[\frac52,\infty)$.

All the results for the stability of the  strange fixed points obtained in the above figures validate the theorem statement. 
\end{proof}

After analyzing the stability of the strange fixed points, we now study the free critical points of the operator, which play an important role in the analysis of parameter lines, dynamical lines, and basins of attraction.
\subsection{Critical points of \texorpdfstring{$R_\beta(x)$}{Rβ(x)}}
\begin{theorem}
The point $x=1$ is a critical point of $R_\beta(x)$ for every $\beta\in\mathbb{R}$. Additionally:
\begin{enumerate}
    \item If $\beta\in(-\infty,1.78]\cup\left(\frac52,4.83\right),$ 
    then the operator has two free critical points $cr_i(\beta)$, $i=1,2$, corresponding to the real roots of
    \begin{equation}
    \begin{aligned}
    p(x)=&(-160+64\beta)x^4+(1760-704\beta)x^3+(-6642+2592\beta)x^2\\&+(9319-3328\beta)x-2927+512\beta.
    \end{aligned}
    \end{equation}
    In particular, at $\beta =\frac{25}{16}$, the free critical points $cr_1(\beta)=1$ and $cr_2(\beta)\approx2.25$. Moreover, at $\beta\approx1.78$, the free critical points $cr_1(\beta)$ and $cr_2(\beta)$ coincide.

    \item If $ \beta\in[4.83,\infty)$,
    then the operator has four free critical points $cr_i(\beta)$ and  $i=1,2,3,4$, given by the real roots of the same polynomial $p(x)$. Moreover, at $\beta\approx4.83$, the free critical points $cr_2(\beta)$ and $cr_3(\beta)$ coincide .
\end{enumerate}
\end{theorem}

\begin{proof}
To obtain the free critical points, we solve
\[
R'_\beta(x)=0.
\]
Since
\begin{equation}
\begin{aligned}
\frac{2}{729}(x-1)^3&(-2927+9319x-6642x^2+1760x^3-160x^4\\&+32\beta(x-4)^2(1-6x+2x^2))=0,
\end{aligned}
\end{equation}
we obtain the triple critical point $x=1$ and the real roots of the polynomial $p(x).$ All the real roots of the above expression are considered the free critical points,  except $x=1$, which corresponds to the root of the rational function.
\end{proof}
Next, we use these free critical points to determine the stable parameter values of $\beta$ for the iterative schemes by constructing  parameter lines.
\subsection{Parameter lines}
To generate the parameter lines, we take the free critical points as initial approximations and analyze their dynamical behavior by determining whether they converge to the root $x=1$ or to a strange attractor. 

To this end, a discretization of $500$ points is defined on the real line. Each point corresponds to a particular value of the parameter $\beta$, and for each value of $\beta$, the operator $R_\beta(x)$ is iterated. If the orbit of the critical point converges to the root $x=1$ within approximately $500$ iterations, then the corresponding value of $\beta$ is colored red. In contrast, if the free critical point converges to a strange attractor or to a periodic orbit, the point is colored black. The blue region represents the divergence of the free critical point.
\begin{figure}[H]
    \centering
\includegraphics[width=0.7\textwidth]{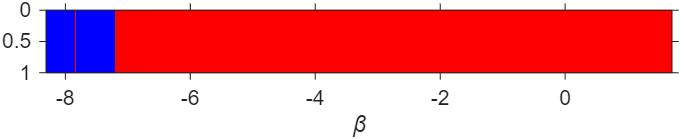}
     \caption{$cr_1(\beta)$ for $\beta \in (-\infty, 1.78]$}
     \label{fig:Linea_param_cr1_menor_1}
\end{figure}

In Figure~\ref{fig:Linea_param_cr1_menor_1}, the behavior of the free critical point $cr_1(\beta)$ defined for $\beta \in (-\infty, 1.78]$ is represented, which shows that for the parameter values  $\beta \in (-7.21\;, \; 1.78]$, the free critical point converges to the root $x=1$.The blue regions correspond to the divergence.

\begin{figure}[H]
    \centering
\includegraphics[width=0.7\textwidth]{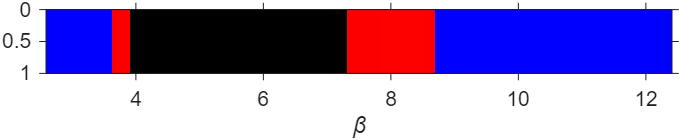}
        \caption{$cr_1(\beta)$ for $\beta \in (\frac{5}{2}, \infty)$}
    \label{fig:Linea_param_cr1_mayor_3}
\end{figure}

Figure~\ref{fig:Linea_param_cr1_mayor_3} shows the behavior of the free critical point $cr_1(\beta)$ when $\beta \in (\frac52,\infty)$. It can be observed that, for $\beta \in (3.64,  3.91) \cup (7.3 , 8.7)$, the free critical point $cr_1(\beta)$ always approaches to the root $x=1$, so it is colored red. In the middle of these regions, the free critical point converges to the strange attractor or to a periodic orbit. Beyond this region, the free critical point diverges.

\begin{figure}[H]
    \centering
\includegraphics[width=0.7\textwidth]{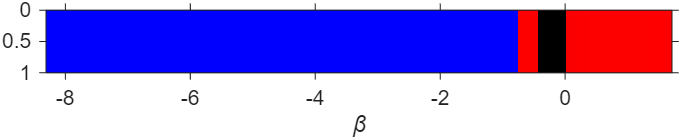}
     \caption{$cr_2(\beta)$ for $\beta \in (-\infty, 1.78]$}
    \label{fig:Linea_param_cr2_menor_1}
\end{figure}

In Figure~\ref{fig:Linea_param_cr2_menor_1}, it can be observed that for approximately $\beta \in [-0.75,-0.4]\cup[0,1.78]$, the free critical point $cr_2(\beta)$ approaches the root $x=1$. For approximately $\beta \in [-0.4,0)$, the free critical point converges to a strange attractor or to an attracting periodic orbit, while for all other values of $\beta$, the free critical point diverges. 

\begin{figure}[H]
    \centering
\includegraphics[width=0.7\textwidth]{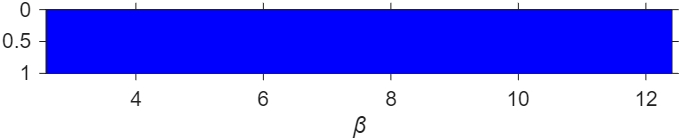}
        \caption{$cr_2(\beta)$ for $\beta \in (\frac{5}{2}, \infty)$}
    \label{fig:Linea_param_cr2_mayor_3}
\end{figure}

Figure~\ref{fig:Linea_param_cr2_mayor_3} shows that for $\beta \in (\frac{5}{2},\infty)$, the free critical point $cr_2(\beta)$ is always divergent.

\begin{figure}[H]
    \centering
\includegraphics[width=0.7\textwidth]{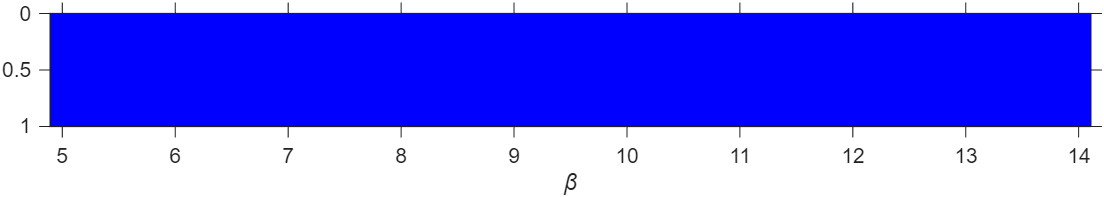}
    \caption{$cr_3(\beta)$ for $\beta \in [4.83, \infty)$}
    \label{fig:Linea_param_cr3_mayor_5}
\end{figure}
\begin{figure}[htbp]
        \centering
\includegraphics[width=0.7\textwidth]{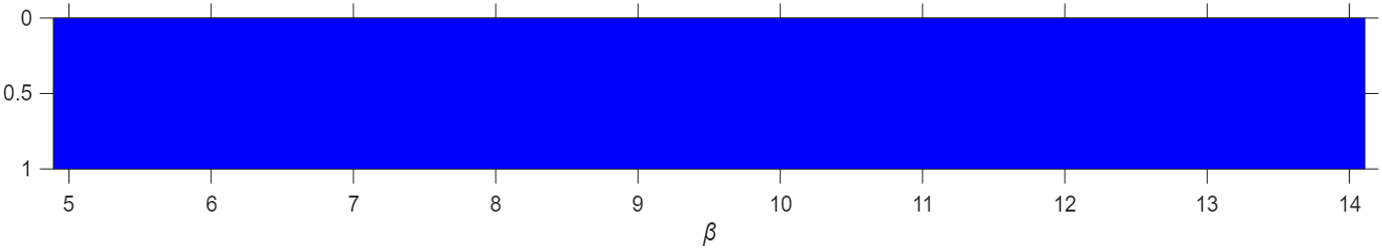}
        \caption{$cr_4(\beta)$ for $\beta \in [4.83, \infty)$}
    \label{fig:Linea_param_cr4_mayor_5}
\end{figure}

Figures~\ref{fig:Linea_param_cr3_mayor_5} and~\ref{fig:Linea_param_cr4_mayor_5} show that the two free critical points $cr_3(\beta)$ and $cr_4(\beta)$ are always divergent for all parameter values where they are defined.

Next, to obtain the global view of the parameter lines, we analyze their unified  behavior. For this purpose, unified parametric lines are generated by considering all free critical points over the intervals in which they are defined. The white regions represent the parametric values for which all free critical points converge to the root $x=1$. Consequently, these regions correspond to the stable iterative schemes within the entire family. In contrast, the black regions show the parametric values for which at least one free critical point diverges or converges to an attracting  strange fixed point or to an attracting periodic orbit.

\begin{figure}[H]
    \centering
\includegraphics[width=0.7\textwidth]{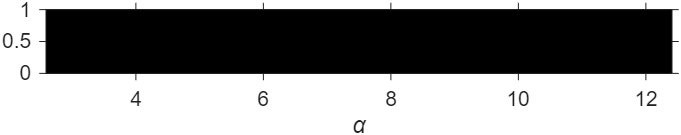}
    \caption{Unified parameter plane for $\beta \in (-\infty, 1.78]$}
    \label{fig:Linea_param_unificada_menor_1}
\end{figure}
\begin{figure}[H]
        \centering
\includegraphics[width=0.7\textwidth]{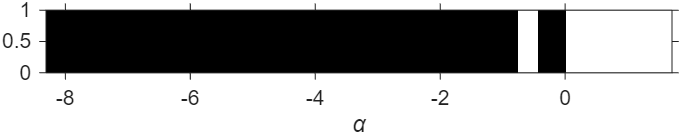}
        \caption{Unified parameter plane for $\beta \in (\frac52, \infty]$}
\label{fig:Linae_param_unificada_mayor3}
\end{figure}

In Figures~\ref{fig:Linea_param_unificada_menor_1} and~\ref{fig:Linae_param_unificada_mayor3}, we can visualize the stable parameter region of the Heun-type family of matrix iterative schemes. we can observe that all free critical points defined in the interval $\beta \in [-0.7448,-0.4376]\cup[0,1.78]$ converge only to the root $x=1$. For the remaining values of $\beta$, at least one free critical point either diverges or converges to a strange attractor. Therefore, the stable region of the parametric family of Heun-type matrix iterative schemes is
\begin{equation}\label{eq:stability region}
\beta \in [-0.7448,-0.4376]\cup[0,1.78].
\end{equation}
Our next step is to determine in detail the behavior of the black regions in Figures~\ref{fig:Linea_param_cr1_mayor_3} and~\ref{fig:Linea_param_cr2_menor_1}. These black regions can give us detailed information about the Long-term behavior of the free critical points. For this purpose, we  check whether,  for these parameter values, the free critical points converge to a strange fixed point, periodic orbits, or  strange attractors. 
\subsection{Bifurcation diagrams}

In parameter lines, the free critical points are used as initial approximations. Some free critical points were always divergent, while others converged to an unknown region and were represented in black. However, only those critical points whose corresponding parameter lines contain black (unstable) regions are considered in this analysis. The main purpose is to determine whether the free critical points converge to strange fixed points, periodic orbits, or strange attractors.

\begin{figure}[htbp]
    \centering
\includegraphics[width=0.7\textwidth]{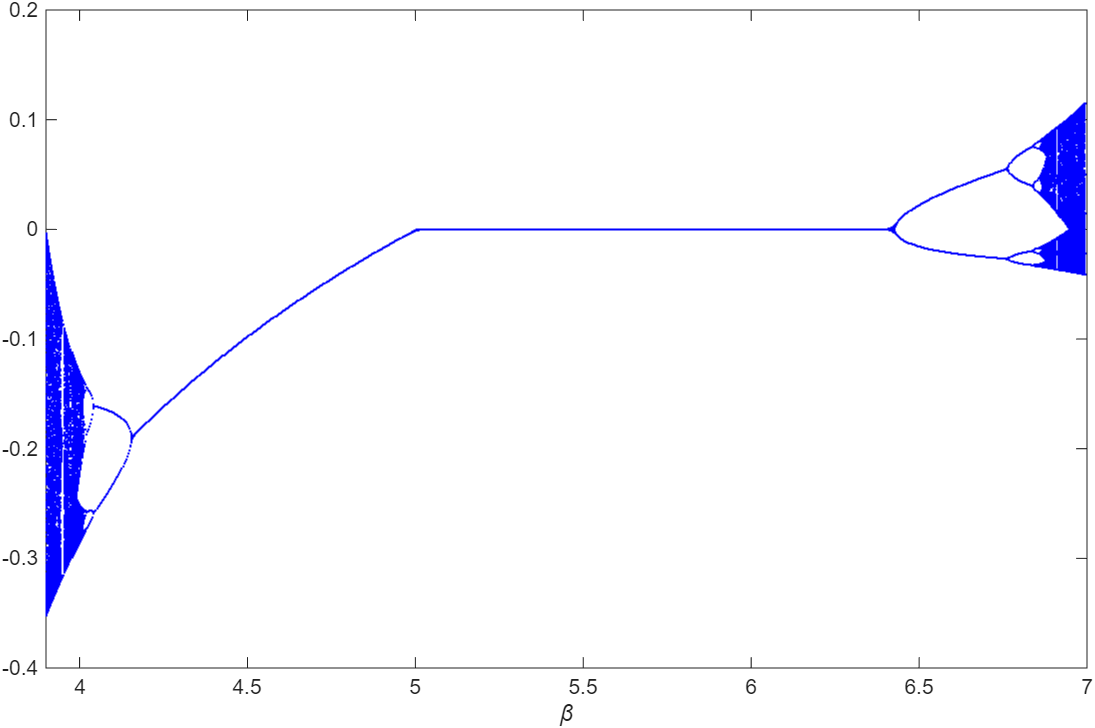}
    \caption{Bifurcation diagram associated with $cr_1(\beta)$.}
        \label{fig:db_cr1}
\end{figure}

Figure~\ref{fig:db_cr1} represents the bifurcation diagram corresponding to the black (unstable) region detected in Figure~\ref{fig:Linea_param_cr1_mayor_3} for approximately  $\beta\in[3.9,\infty[$. A qualitative change near $\beta\approx5$ can be observed in a single branch over the interval  approximately $\beta\in[4.3,6.43]$. This interval coincides with the parameter values for which the strange fixed points $x=0$ and $x=ex_1(\beta)$ are attracting. Therefore, we can conclude that the free critical point $cr_1(\beta)$ converges to one of these strange attractors throughout this region. The qualitative change at $\beta=5$ represents the change of the attracting strange fixed point; that is, initially, it converges toward the strange fixed point  $x=ex_1(\beta)$, and then toward the point  $x=0$.

Outside this interval, a sequence of period-doubling bifurcations can be observed. A period-two orbit converts into a period-four orbit, which subsequently  converts into a period-eight orbit  and higher-order periodic orbits. Finally, these cascades move toward chaotic dynamics, represented by the dense blue bands appearing near $\beta\in [3.95,4]$ and for $\beta\gtrsim6.8$.

To characterize these chaotic regions, the corresponding attractors are displayed in Figure~\ref{fig:strange_Atr_cr1_left}.

\begin{figure}[H]
    \centering
\includegraphics[width=0.7\textwidth]{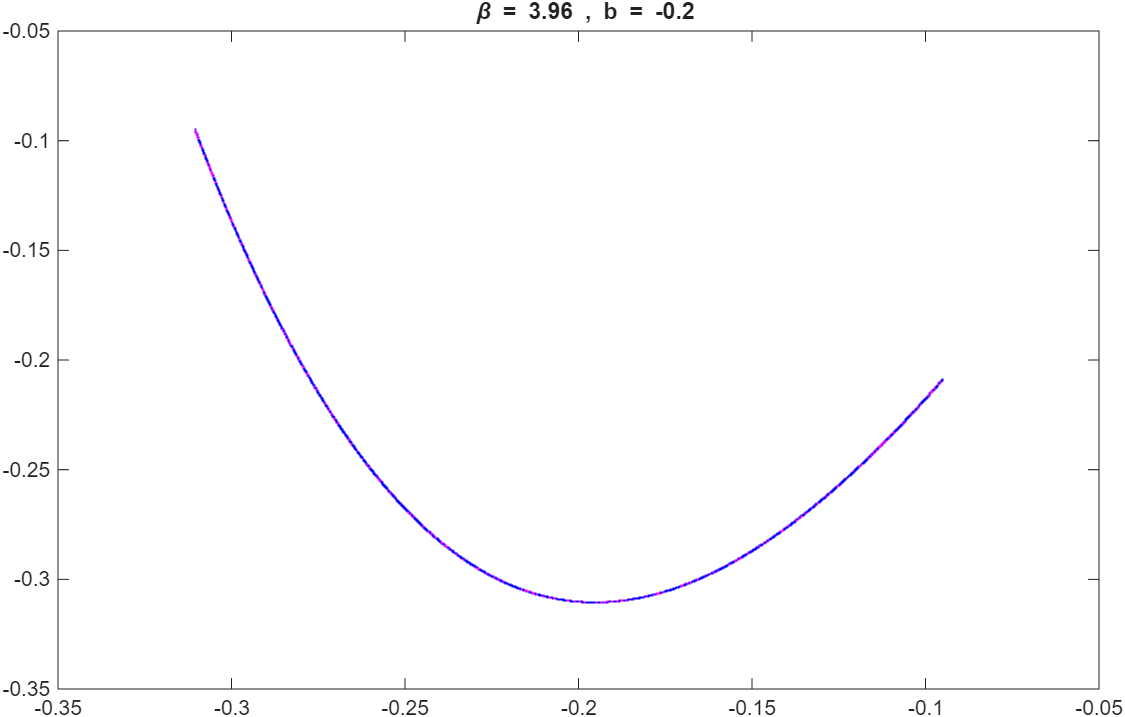}
    \caption{$\beta=3.96$}
    \label{fig:strange_Atr_cr1_left}
\end{figure}

Figure~\ref{fig:strange_Atr_cr1_left} represents the asymptotic behavior of trajectories generated from $1000$ initial conditions. To eliminate the transient effects, for each orbit, the first $500$ iterates are discarded, and the following $400$ iterates are represented in blue, while the last $100$ iterates are colored magenta.

For $\beta=3.96$, we can observe the confined orbit in a  smooth invariant curve. The trajectory is bounded and exhibits a non-periodic structure. It provides numerical evidence of the existence of strange attractors for these parameter values.

\begin{figure}[H]
    \centering
\includegraphics[width=0.7\textwidth]{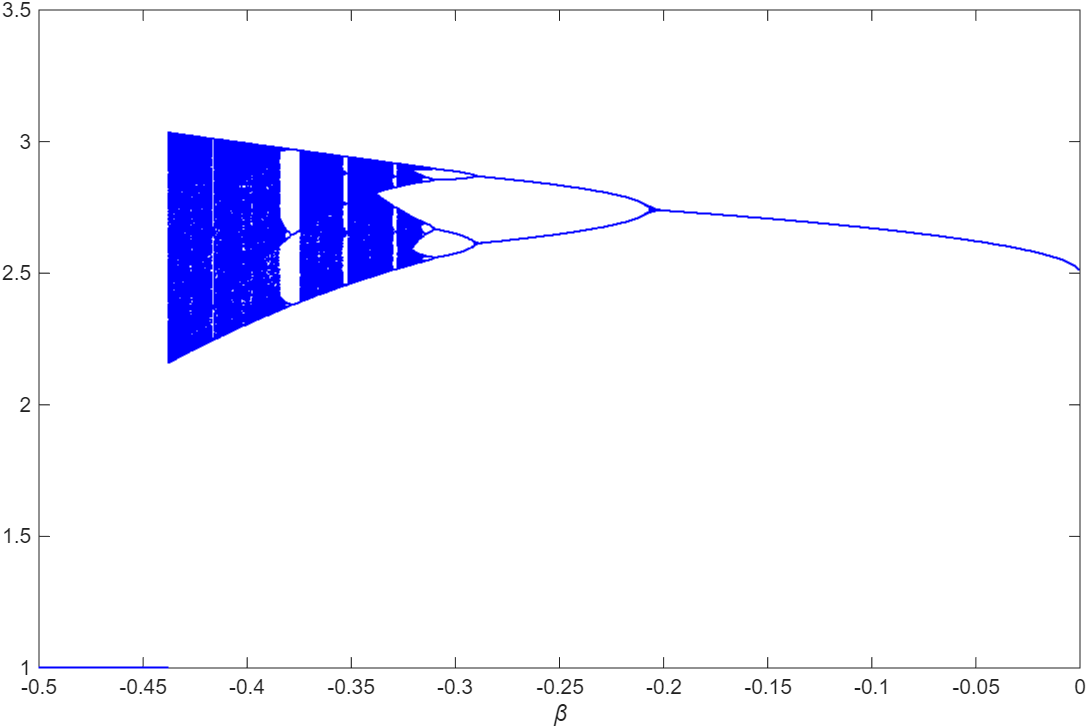}
    \caption{Bifurcation diagram associated with $cr_2(\beta)$.}
        \label{fig:db_cr2}
\end{figure}

Figure~\ref{fig:db_cr2} represents the bifurcation diagram associated with the black(unstable) region identified in Figure~\ref{fig:Linea_param_cr2_menor_1} for $\beta\in[-0.5,0[$. A single branch is observed for approximately $\beta\in]-0.21,0[$, which coincides with the interval where the strange fixed point $ex_2(\beta)$ is attracting. Therefore, the free critical point $cr_2(\beta)$ converges to this strange fixed point attractor for the corresponding parameter values.

On the left side of the single branch, a classical period-doubling cascade appears, which leads to periodic orbits of increasing order and eventually to chaotic behavior. Therefore, we  can visualize a chaotic region approximately in the interval $\beta\in[-0.43,-0.32]$. Phase space is shown in Figure~\ref{fig:strange_Atr_cr2_2} to verify the presence of chaotic attractors within this region.

\begin{figure}[H]
    \centering
\includegraphics[width=0.7\textwidth]{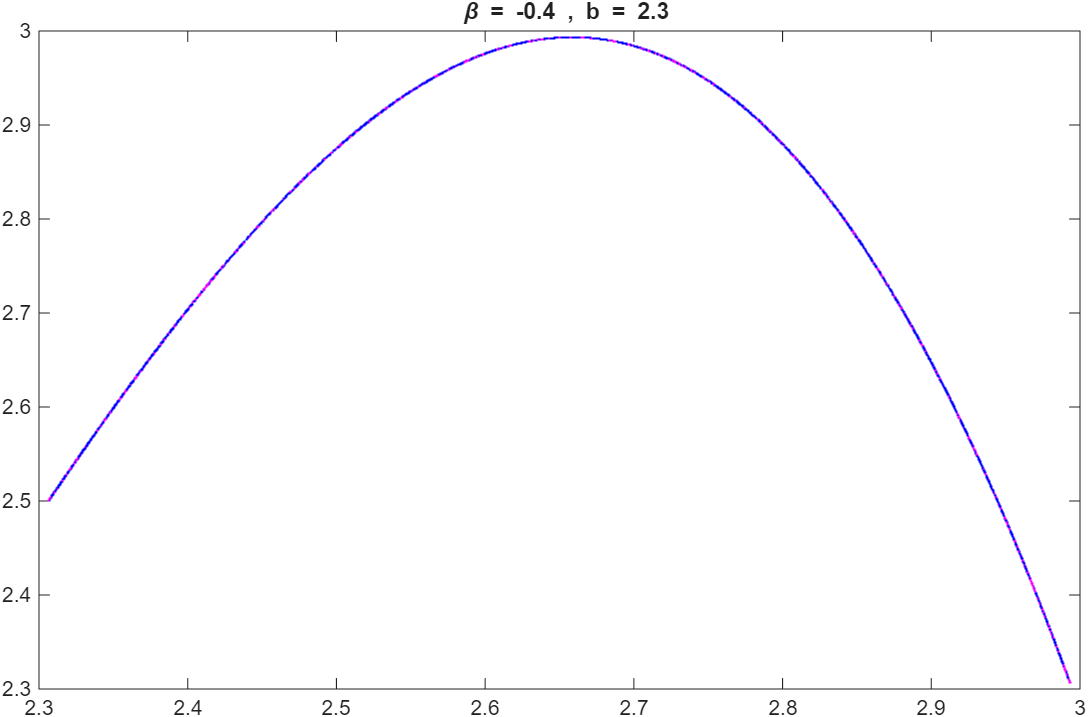}
    \caption{$\beta=-0.40$}
        \label{fig:strange_Atr_cr2_2}
\end{figure}

Figure~\ref{fig:strange_Atr_cr2_2} represents the asymptotic dynamics corresponding to $\beta=-0.40$. Again, the trajectory remains confined to smooth invariant curves and do not converge to fixed points or to finite periodic cycles. The existence of bounded non-periodic motion indicates the presence of a strange attractor. 

Next, we analyze different members of the Heun-type matrix iterative family by employing dynamical lines corresponding to the parameter values $\beta$. These dynamical lines allow us to visualize the basins of attraction associated with each iterative scheme.
\subsection{Dynamical lines}
 In dynamical lines, each point is considered an initial approximation. Moreover, we represent each point by its convergence toward different attractors using different colors. The dynamical lines are generated by using a maximum of $40$ iterations for each initial approximation and a discretization of $400$ points. In dynamical lines, the points converging to the root $x=1$ are represented in orange, while the diverging points are shown in blue. The points that converge to a strange attractor are represented in green, whereas the points that converge to an attracting periodic point are colored black. 
We divide the Heun-type matrix family into three regions. Then, we verify whether the theoretical results obtained previously are consistent with the observed dynamical behavior.\\
According to the Theorem~\ref{theo:convergence_heun4} and equation \eqref{eq:stability region}, the interval
\[
\beta \in [-0.7448,-0.4376]\cup\left[0,\frac{25}{16}\right[
,\]
corresponds to the region where the family of iterative schemes is stable, but convergence is not assured. On the other hand, for
\[
\beta \in \left[\frac{25}{16},1.78\right],
\]
the family is both stable and convergent. The remaining parametric values exhibit unstable behavior.

 \begin{figure}[htbp]
    \centering
\includegraphics[width=0.7\textwidth]{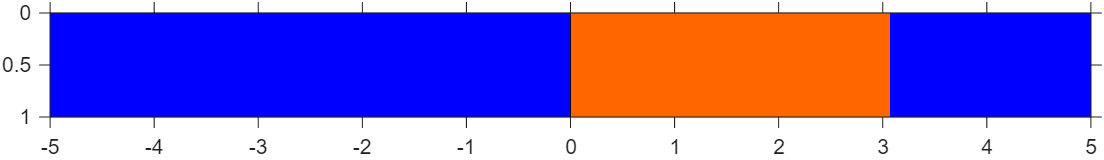}
    \caption{Dynamical line of stable and convergent parameter value $\beta=\frac{25}{16}$.}
        \label{fig:Linea din estable y convergente 2}
\end{figure}
Figure~\ref{fig:Linea din estable y convergente 2} illustrate the dynamical lines corresponding to values of $\beta$ belonging to the stable and convergent regions. It can be observed that there is only one basin of attraction associated with the root $x=1$, while the remaining points diverge.

\begin{figure}[htbp]
    \centering
\includegraphics[width=0.7\textwidth]{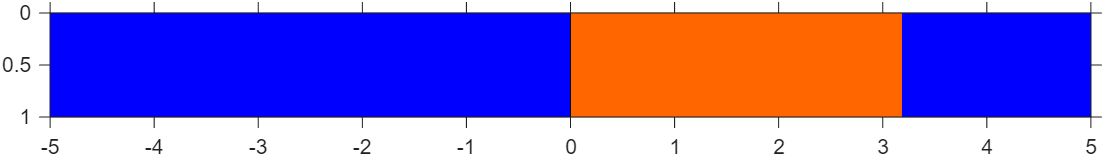}
    \caption{Dynamical lines of stable but non-assured convergent parameter values $\beta=1$.}
        \label{fig:Linea din estable y no-convergente 2}
\end{figure}
Figure~\ref{fig:Linea din estable y no-convergente 2} represents the dynamical lines corresponding to values of $\beta$ that belong to the stable but non-assured convergent region. In this case, there is only one basin of attraction associated with the root $x=1$, and the interval corresponding to this basin is larger than that in Figure~\ref{fig:Linea din estable y convergente 2}. Although these parameter values lie outside the convergence region established in Theorem \ref{theo:convergence_heun4}, the methods still exhibit satisfactory dynamical behavior.
\begin{figure}[htbp]
    \centering
\includegraphics[width=0.7\textwidth]{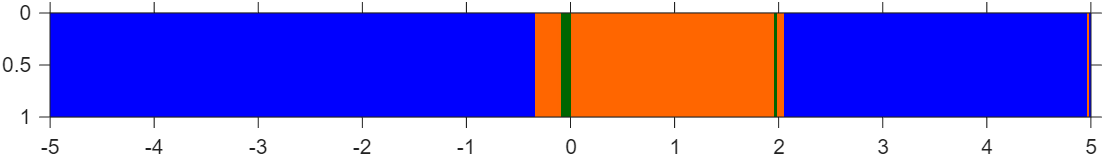}
    \caption{Dynamical lines of non-stable parameter values $\beta=5.004$.}
        \label{fig:Linea din no estable 1}
\end{figure}

Finally, Figure~\ref{fig:Linea din no estable 1} corresponds to values of $\beta$ belonging to the unstable region, where some strange fixed points become attracting. It can be observed that two types of basins of attraction appear along with the divergent blue region: one associated with the root $x=1$ and another associated with a strange attracting fixed point represented in green.

Finally, some numerical experiments are performed on different sizes of random matrices to study the numerical behavior of the proposed iterative schemes.
\section{Numerical Experiments}\label{sec:numerical}
Throughout this section, numerical experiments are presented for the computation of the standard inverse, the Moore-Penrose inverse, and the Drazin inverse by employing the proposed iterative schemes. To evaluate their  performance, the error between two successive iterates $\|X_{n+1}-X_n\|_2,$
and the residual error $\|F(X_{n+1})\|=\|I-AX_{n+1}\|_2,$ 
are computed at each iteration for each type of inverse. In the case of generalized inverses, the properties of the Moore-Penrose and Drazin inverses introduced in Section \ref{sec:preliminaries} are verified numerically. Since all matrix norms  give the same qualitative information, we use only the Euclidean norm in the  numerical tests.

The stopping criterion used in the experiments  for the standard inverse and the Moore-Penrose inverse is 
\[
\|X_{n+1}-X_n\|_2 < 10^{-9}
\quad \text{and} \quad
\|I-AX_{n+1}\|_2 < 10^{-9},
\]
using double-precision arithmetic.  For the Drazin inverse, the stopping criterion is 
\[
\|X_{n+1}-X_n\|_2 < 10^{-9}.
\]
To determine the convergence order of iterative schemes, the ACOC is computed  by
\[
\mathrm{ACOC}
=
\frac{
\ln \left(
\dfrac{\|X_{n+1}-X_n\|}
{\|X_n-X_{n-1}\|}
\right)
}{
\ln \left(
\dfrac{\|X_n-X_{n-1}\|}
{\|X_{n-1}-X_{n-2}\|}
\right)
}, \qquad n \geq 2.
\]
This technique provides a practical numerical estimation of the convergence order in matrix iterative schemes (see \cite{Cordero_2007}). The proposed Heun-type iterative schemes ($BR(\beta)$) are compared with the classical Newton-Schulz method and the Homeier scheme. The theoretical convergence order of the Heun-type matrix iterative methods is established in Theorem~\ref{theo:convergence_heun4},  while the quadratic convergence order of the Newton-Schulz method and the cubic convergence order of the Homeier methods are given in \cite{Schulz_1933, Kansal_2023}, respectively. For CPU-time, the results show an average of $30$ execution for all iterative schemes. All the results were carried out in MATLAB R2025a on an  Intel Core i3-N305 processor running at 1.80 GHz with 8 GB RAM.

For approximating the standard inverse, we considered the convergence region and stability region  obtained in Theorem \ref{theo:convergence_heun4} and in equation \eqref{eq:stability region}, respectively.
\begin{example}\label{EX:Sq.Mat_heun4_50}
Let $A\in\mathbb{C}^{N\times N}$ be a randomly generated non-singular square matrix with $N\in \{50, 100, 500\}$. The initial approximation is chosen as
\[
X_0=\frac{A^*}{\|A\|_2^2}.
\]
The objective is to approximate the inverse of $A$.
\end{example}

\begin{table}[htbp]
\centering
\caption{Results for approximating the inverse of the random matrix $A\in\mathbb{C}^{50\times 50}$ in Example~\ref{EX:Sq.Mat_heun4_50}.}
\label{tab:Sq.Mat_heun4_50}
\begin{tabular*}{\textwidth}{@{\extracolsep\fill}lllllll@{}}\toprule 
Method & $\beta$ & Iter. & $\|I-AX_{n+1}\|_2$ & $\|X_{n+1}-X_n\|_2$ & ACOC&CPU-time(s) \\
\midrule
Newton--Schulz&$--$  & $25$ & $2.36\times10^{-14}$ & $7.49\times10^{-14}$ & $2.0014$& $0.0208$\\
Homeier  &$--$    & $14$ & $2.52\times10^{-14}$ & $3.09\times10^{-10}$ & $3.5620$ &$0.0167$\\
\midrule
\multicolumn{7}{c}{Values of $\beta$ for which both convergence and stability are guaranteed} \\
\midrule
$BR(\beta)$ & $1.75$ & $11$ & $9.18\times10^{-15}$ & $8.19\times10^{-14}$ & $5.0827$&$0.0155$ \\
$BR(\beta)$ & $1.68$ & $11$ & $1.01\times10^{-15}$ & $8.76\times10^{-14}$ & $4.9774$&$0.0159$ \\
\midrule
\multicolumn{7}{c}{Values of $\beta$ for which stability is guaranteed but not convergence} \\
\midrule
$BR(\beta)$ & $0.5$ & $10$ & $9.22\times10^{-15}$ & $7.27\times10^{-14}$ & $3.5117$&$0.0146$ \\
$BR(\beta)$ & $-0.5$ & $10$ & $9.53\times10^{-15}$ & $7.88\times10^{-14}$ & $3.9395$&$0.0144$ \\
\midrule
\multicolumn{7}{c}{Values of $\beta$ for which neither convergence nor stability is guaranteed} \\
\midrule
$BR(\beta)$ & $-30$ & $50$ & $7.36\times10^{95}$ & $5.26\times10^{95}$ & $--$&$0.2459$ \\
$BR(\beta)$ & $50$ & $50$ & $7.98\times10^{240}$ & $1.45\times10^{240}$ & $--$&$0.2482$ \\
\midrule
\multicolumn{7}{c}{Values of $\beta$ for which the Heun-type family attains fifth-order convergence} \\
\midrule
$BR(\beta)$ & $\dfrac{25}{16}$ & $11$ & $1.07\times10^{-14}$ & $7.63\times10^{-14}$ & $5.3062$&$0.0148$ \\
\bottomrule
\end{tabular*}%

\end{table}


\begin{table}[htbp]
\centering
\caption{Results for approximating the inverse of the random matrix $A\in\mathbb{C}^{100\times 100}$ in Example~\ref{EX:Sq.Mat_heun4_50}.}
\label{tab:Sq.Mat_heun4_100}
\begin{tabular*}{\textwidth}{@{\extracolsep\fill}lllllll@{}}\toprule 
Method & $\beta$ & Iter. & $\|I-AX_{n+1}\|_2$ & $\|X_{n+1}-X_n\|_2$ & ACOC &CPU-time(s)\\
\midrule
Newton--Schulz & $--$ & $32$ & $2.69\times10^{-13}$ & $5.05\times10^{-12}$ & $2.0107$& $0.1775$\\
Homeier        & $--$ & $19$ & $2.82\times10^{-13}$ & $2.80\times10^{-12}$ & $3.0331$& $0.1063$ \\
\midrule
\multicolumn{7}{c}{Values of $\beta$ for which both convergence and stability are guaranteed} \\
\midrule
$BR(\beta)$ & $1.75$ & $14$ & $6.59\times10^{-14}$ & $2.95\times10^{-12}$ & $5.2136$&$0.1097$ \\
$BR(\beta)$ & $1.68$ & $14$ & $6.57\times10^{-14}$ & $2.69\times10^{-12}$ & $4.9925$&$0.1090$ \\
\midrule
\multicolumn{7}{c}{Values of $\beta$ for which stability is guaranteed but convergence is not assured} \\
\midrule
$BR(\beta)$ & $0.5$ & $13$ & $6.77\times10^{-14}$ & $2.95\times10^{-12}$ & $3.8020$ &$0.1036$\\
$BR(\beta)$ & $-0.5$ & $12$& $6.75\times10^{-14}$ & $3.10\times10^{-12}$ & $3.4237$&$0.0957$ \\
\midrule
\multicolumn{7}{c}{Values of $\beta$ for which neither convergence nor stability is guaranteed} \\
\midrule
$BR(\beta)$ & $-30$ & $50$ & $3.64\times10^{96}$ & $1.01\times10^{96}$ & $--$ &$0.9957$\\
$BR(\beta)$ & $50$ & $50$ & $3.78\times10^{180}$ & $4.84\times10^{179}$ & $--$&$0.9479$\\
\midrule
\multicolumn{7}{c}{Values of $\beta$ for which the Heun-type family attains fifth-order convergence} \\
\midrule
$BR(\beta)$ & $\dfrac{25}{16}$ & $14$ & $6.33\times10^{-14}$ & $3.22\times10^{-12}$ & $5.2576$&$0.1105$ \\
\bottomrule
\end{tabular*}%
\end{table}


\begin{table}[htbp]
\centering
\caption{Results for approximating the inverse of the random matrix $A\in\mathbb{C}^{500\times 500}$ in Example~\ref{EX:Sq.Mat_heun4_50}.}
\label{tab:Sq.Mat_heun4_500}
\begin{tabular*}{\textwidth}{@{\extracolsep\fill}lllllll@{}}\toprule 
Method & $\beta$ & Iter. & $\|I-AX_{n+1}\|_2$ & $\|X_{n+1}-X_n\|_2$ & ACOC & CPU-time(s)\\
\midrule
Newton--Schulz & $--$ & $34$ & $4.21\times10^{-13}$ & $7.72\times10^{-11}$ & $2.0192$&$5.2936$ \\
Homeier        & $--$ & $20$ & $4.31\times10^{-13}$ & $1.41\times10^{-12}$ & $3.0728$&$4.5263 $ \\
\midrule
\multicolumn{7}{c}{Values of $\beta$ for which both convergence and stability are guaranteed} \\
\midrule
$BR(\beta)$ & $1.75$ & $15$ & $2.89\times10^{-15}$ & $1.025\times10^{-4}$ & $4.8422$ &$4.8015$\\
$BR(\beta)$ & $1.68$ & $15$ & $7.08\times10^{-14}$ & $1.38\times10^{-12}$ & $4.7051$&$4.5791$\\
\midrule
\multicolumn{7}{c}{Values of $\beta$ for which stability is guaranteed but convergence is not} \\
\midrule
$BR(\beta)$ & $0.5$ & $13$ & $7.00\times10^{-14}$ & $5.8469\times10^{-11}$ & $5.2057$&$4.2979$ \\
$BR(\beta)$ & $-0.5$ & $13$ & $6.99\times10^{-14}$ & $1.3887\times10^{-12}$ & $3.8568$&$4.1226$ \\
\midrule
\multicolumn{7}{c}{Values of $\beta$ for which neither convergence nor stability is guaranteed} \\
\midrule
$BR(\beta)$ & $-30$ & $50$ & $3.63\times10^{96}$ & $2.5327\times10^{95}$ & $--$&$36.7232$ \\
$BR(\beta)$ & $50$ & $50$ & $2.34\times10^{82}$ & $1.2868\times10^{81}$ & $--$&$32.1419$ \\
\midrule
\multicolumn{7}{c}{Values of $\beta$ for which the Heun-type family attains fifth-order convergence} \\
\midrule
$BR(\beta)$ & $\dfrac{25}{16}$ & $15$ & $6.99\times10^{-14}$ & $1.4006\times10^{-12}$ & $5.1475$&$4.6069$ \\
\bottomrule
\end{tabular*}%
\end{table}

Tables~\ref{tab:Sq.Mat_heun4_50}, \ref{tab:Sq.Mat_heun4_100}, and \ref{tab:Sq.Mat_heun4_500} present the numerical results obtained using the Newton-Schulz, Homeier, and Heun-type iterative schemes for approximating the inverse of randomly generated non-singular matrices.

The  numerical results confirm the theoretical convergence orders: quadratic convergence for the Newton--Schulz method and cubic convergence for the Homeier method. Some interesting results from the fourth-order Heun-type family of iterative schemes can be analyzed.

The values of $\beta$ for which both convergence and stability are guaranteed show favorable behavior of the corresponding iterative schemes for all sizes of matrices. The residual error and the error between successive iterates are significantly low. Moreover, the order of convergence is also increased. The number of iterations increases with the matrix size but remains lower than that of the schemes with a lower convergence order. The CPU-time for these members of the Heun-type family is lower than that of the Newton-Schulz scheme. 

Similar results are obtained for the case where stability is guaranteed, but convergence is not assured. For different members of the Heun-type matrix family, we obtain fourth-order convergence even though theoretical convergence is not guaranteed.

For the parameter values where neither stability nor convergence is guaranteed, we observe that the members of the family do not converge to the solution within the prescribed iteration limit. Also, they exhibit divergent and unstable behavior, as indicated by the increasing errors. 

For $\beta=\frac{25}{16}$, we obtained theoretically a fifth-order Heun-type matrix iterative scheme. Hence, the numerical results confirm that the fifth-order convergence is maintained. Moreover, the errors between the iterations and the residual error remain sufficiently small.

CPU time for the iterative schemes where either convergence or stability is maintained is lower than the Newton-Schulz for almost all iterative schemes.

\begin{example}\label{EX:Rec.Mat_Heun4_Moorep}
Let $A\in\mathbb{C}^{M\times N}$ be a randomly generated rectangular matrix for different values of $M$ and $N$. The initial approximation for all methods is chosen as
\[
X_0=0.9\frac{A^*}{\|A\|_2^2}
.
\]
\end{example}

Similar to the previous section, we will denote the residual error, the error between the iterate and the conditions of the Moore-Penrose inverse, in  the following way:
\[
F(X_{n+1})=\|I-AX_{n+1}\|_2,
\qquad
Error=\|X_{n+1}-X_n\|_2,
\] with
\begin{align*}
PM1 &= \|AX_{n+1} A-A\|_2,\\
PM2 &= \|X_{n+1} AX_{n+1}-X_{n+1}\|_2,\\
PM3 &= \|X_{n+1} A-(X_{n+1} A)^*\|_2,\\
PM4 &= \|AX_{n+1}-(AX_{n+1})^*\|_2.
\end{align*}
\begin{table}[H]
\centering
\caption{Results for approximating the Moore--Penrose inverse of a random matrix $A \in \mathbb{C}^{10 \times 20}$ with $\beta\in \{2, 2.5, \frac{25}{16}\}.$}
\label{tab:Moore_inverse_10n20_heun4}
\resizebox{\textwidth}{!}{%
\begin{tabular}{ccccccccccccc}
\toprule 
Method &$\beta$&M&N& Iter. & $F(X_{n+1})$ & $Error$ & $PM1$ & $PM2$ & $PM3$ & $PM4$ & ACOC&CPU-time (s) \\
\midrule
NS &$--$&$10$&$20$& $14$ & $5.5\times 10^{-16}$ & $5.1 \times 10^{-16}$ &$2.6\times 10^{-14}$& $3.8\times 10^{-14}$&$2.1\times 10^{-15}$&$8.7\times 10^{-16}$&$2.0024$&$0.0067$ \\
Homeier&$--$&$10$&$20$& $8$  & $7.7 \times 10^{-16}$ & $1.6 \times 10^{-11}$ &$8.9\times 10^{-12}$&$1.6\times 10^{-11}$&$2.3\times 10^{-15}$&$8.1\times 10^{-16}$& $3.5380$ &$0.0073$\\
$BR(\beta)$&$2$ &$10$&$20$& $6$ & $2.0\times10^{-14}$ & $1.4\times10^{-14}$ &
$7.2\times10^{-14}$ & $1.1\times10^{-14}$ &
$6.9\times10^{-14}$ & $1.6\times10^{-14}$ & $4.3467$&$0.0034$ \\
\midrule
NS & $--$&$100$&$200$&$18$ & $4.4\times 10^{-15}$ & $6,2 \times 10^{-16}$ &$1.0\times 10^{-13}$& $5.0\times 10^{-16}$&$2.4\times 10^{-14}$&$5.5\times 10^{-15}$&$2.0033$&$0.2218$ \\
Homeier&$--$&$100$&$200$& $11$ & $4.2 \times 10^{-15}$ & $6.4 \times 10^{-16}$ &$1.3\times 10^{-13}$&$5.2\times 10^{-16}$&$2.2\times 10^{-14}$&$6.9\times 10^{-15}$& $3.0443$&$0.0962$\\
$BR(\beta)$ &$2.5$&$100$&$200$ & $9$ & $4.5\times10^{-14}$ & $2.1\times10^{-14}$ &
$1.9\times10^{-12}$ & $1.4\times10^{-14}$ &
$4.3\times10^{-13}$ & $5.8\times10^{-14}$ & $4.5529$&$0.0936$ \\
\midrule
NS &$--$&$500$&$600$& $23$ & $3.7\times 10^{-14}$ & $3.8 \times 10^{-10}$ &$3.3\times 10^{-10}$& $3.9\times 10^{-10}$&$1.7\times 10^{-13}$&$3.9\times 10^{-14}$&$2.0465$&$3.4104$ \\
Homeier&$--$&$500$&$600$& $14$ & $3.0 \times 10^{-15}$ & $3.0 \times 10^{-14}$ &$1.3\times 10^{-12}$&$2.3\times 10^{-15}$&$1.8\times 10^{-13}$&$3.7\times 10^{-14}$& $3.1008$&$2.4954$\\
$BR(\beta)$ &$\frac{25}{16}$&$500$&$600$& $10$ & $8.8\times10^{-14}$ & $1.2\times10^{-13}$ &
$2.2\times10^{-11}$ & $6.3\times10^{-14}$ &
$5.6\times10^{-12}$ & $1.6\times10^{-13}$ & $10.5721$&$3.0984$ \\
\bottomrule
\end{tabular}%
}
\end{table}
Table~\ref{tab:Moore_inverse_10n20_heun4} shows the numerical results obtained using Newton-Schulz, Homeier, and the members $BR(2)$, $BR(2.5)$, and $BR\Big(\frac{25}{16}\Big)$ of the Heun-type matrix iterative family. All these iterative schemes satisfy the convergence order established in Theorem~\ref{theo:Moore_pen_Heun4} and in \cite{Schulz_1933, Kansal_2023}. However, for $BR\Big(\frac{25}{16}\Big)$, the convergence order is not maintained and increases significantly. Moreover, the numerical results indicate that, apart from obtaining small residual and iteration errors, the computed approximations satisfy the Moore-Penrose conditions with high accuracy. Additionally, since the Heun-type matrix iterative schemes have a higher convergence order than the Newton-Schulz and Homeier methods, they require fewer iterations to converge to the solution. No considerable difference can be observed for different iterative schemes depending on the size of the matrix. 

\begin{example}\label{Ex:Drazin_particular_Heun4}
Let $A\in\mathbb{R}^{6\times6}$ be the matrix
\[
A=
\begin{pmatrix}
1 & -1 & 0 & 0 & 0 & 0\\
-1 & 1 & 0 & 0 & 0 & 0\\
-1 & -1 & 1 & -1 & 0 & 0\\
-1&-1&-1&1&0&0\\
-1 & -1 & -1 & 0 & 2 & -1\\
-1 & -1 & 0 & -1 & -1 & 2
\end{pmatrix}.
\]
The initial approximation is chosen as
\[
X_0=\frac{2A^k}{\|A\|^{k+1}},
\]
with index $k=\operatorname{ind}(A)=2$.

The exact Drazin inverse of $A$,  defined in \cite{Artidiello_2020}, is as follows
\[
A^D=
\begin{pmatrix}
\frac14 & -\frac14 & 0 & 0 & 0 & 0\\

-\frac14 & \frac14 & 0 & 0 & 0 & 0\\

0 & 0 & \frac14 & -\frac14 & 0 & 0\\

0 & 0 & -\frac14 & \frac14 & 0 & 0\\

0 & 0 & -\frac{5}{12} & -\frac{7}{12} & \frac23 & \frac13\\

0 & 0 & -\frac{7}{12} & -\frac{5}{12} & \frac13 & \frac23
\end{pmatrix}.
\]
\end{example}
Similar to the previous section, to verify the Drazin inverse conditions, the following quantities are computed:
\[
Error=\|X_{n+1}-X_n\|_2,
\]
 with
\begin{align*}
D1 &= \|A^kX_{n+1}A-A^k\|,\\
D2 &= \|X_{n+1}AX_{n+1}-X_{n+1}\|,\\
D3 &= \|AX_{n+1}-X_{n+1}A\|.
\end{align*}
where $k$ is the index of matrix A.
\begin{table}[H]
\centering
\caption{Results for approximating the Drazin inverse of the matrix in Example~\ref{Ex:Drazin_particular_Heun4}.}
\label{tab:Drazin_Heun4_matrix_given}
\resizebox{\textwidth}{!}{%
\begin{tabular}{cccccccc}\toprule 
Method & Iter. & $Error$ & $D1$ & $D2$ & $D3$ & ACOC& CPU-time (s) \\
\midrule
Newton--Schulz & $10$ & $7.7\times10^{-14}$ &
$7.6\times10^{-14}$ &
$7.7\times10^{-14}$ &
$8.0\times10^{-16}$ & $2.0113$&$0.00055$ \\
Homeier & $6$ & $2.5\times10^{-11}$ &
$2.5\times10^{-11}$ &
$2.5\times10^{-11}$ &
$1.2\times10^{-15}$ & $3.5613$&$0.00036$ \\
$BR(\beta)$ & 5 & $6.6\times10^{-15}$ &
$6.6\times10^{-14}$ &
$4.4\times10^{-15}$ &
$5.4\times10^{-15}$ & $6.15$&$0.00043$ \\
\bottomrule
\end{tabular}%
}
\end{table}

Table~\ref{tab:Drazin_Heun4_matrix_given} illustrates the numerical results obtained using the member $BR\Big(\frac{25}{16}\Big)$ of the  Heun-type matrix iterative family.
the Heun-type matrix iterative schemes converge to the solution in fewer iterations due to their higher convergence order. Moreover, the numerical results indicate that the proposed Heun-type iterative scheme attains a computational convergence order close to six..
It can also be observed that the results obtained by ACOC are consistent with the theoretical results for the Newton-Schulz and Homeier methods.  The overall numerical results, such as the error between the iterates and the error corresponding to the properties of the Drazin inverse, are favorable for the proposed iterative schemes. No considerable difference in CPU-time can be observed.

\section{Conclusions}\label{sec5}
In this manuscript, a family of fourth-order root-finding schemes is developed by incorporating a suitable weight function into the third-order Heun method. Afterward, its matrix framework is established for approximating inverse and generalized inverse matrices. Subsequently, the convergence and convergence order are proved for standard inverse and generalized inverses, such as the Moore-Penrose inverse and the Drazin inverse. Consequently, some members of the Heun-type family converge with fourth-order convergence. Meanwhile, for a particular member, we obtain fifth-order convergence.

Additionally, real dynamics is employed to analyze the dynamical behavior of the Heun-type family. As a result, stable iterative schemes are achieved from the entire family. The unstable region is studied employing bifurcation diagrams. Finally, some numerical experiments are performed on different sizes of random matrices for approximating the standard and generalized inverses. These experiments show the favorable behavior of the iterative schemes from the proposed family, where both convergence and stability are guaranteed, as well as when either convergence or stability is guaranteed.

\subsection*{Author Contributions}

All the authors contributed equally to this work.

\subsection*{Acknowledgments}

\subsection*{Financial Disclosure}

\subsection*{Conflicts of Interest}

The authors declare no conflicts of interest.

\bibliographystyle{ieeetr}
\bibliography{References}

\end{document}